\documentclass{amsart}
\usepackage[dvips]{color}
\usepackage{graphicx}
\usepackage{epsfig}
\usepackage{enumerate}
\usepackage{enumitem}
\usepackage{color}
\usepackage{latexsym}
\usepackage{amssymb}	
\usepackage{amsmath}
\usepackage{amsthm}
\usepackage{amscd}
\usepackage{latexsym}
\usepackage{amssymb}
\usepackage{amsmath}
\usepackage{amsthm}
\usepackage{amscd}
\theoremstyle{plain}
\newtheorem{thm}{Theorem}[section]
\newtheorem{lem}[thm]{Lemma}
\newtheorem{defin}[thm]{Definition}
\newtheorem{prop}[thm]{Proposition}
\theoremstyle{definition}

\newtheorem{rem}[thm]{Remark}

\newcommand{\future}[1]{{}}  

\usepackage{amsmath,amsfonts,amssymb,amsthm,epsfig,float, color,tikz}
\usepackage[margin=1.5in]{geometry}

\newcommand{\emul}{\stackrel{{}_\ast}{\asymp}}

\title{Generic measures for translation surface flows}%[Hdim nue]

\author{Howard Masur}

\email{masurhoward@gmail.com}

\address{Department of Mathematics, University of Chicago, 5734 S. University Avenue, Room 208C, Chicago, IL 60637, USA}

\begin{document}
%\titlerunning{Counting generic measures}

 %\author
%{Howard Masur}
%
%\date{19 February  2020}

\maketitle
%
%\address

%{Apisa: Yale University Department of Mathematics; \email{paul.apisa@yale.edu}
%\and
%{University of Chicago Department of Mathematics; \email{masurhoward@gmail.com}}
%

%\subjclass{Primary 37F30; Secondary 32G15}

%\maketitle
%\begin{center} \textit{Dedicated to the memories  of William Veech and Jean-Christophe Yoccoz}
%\end{center}
\begin{abstract}
We consider straight line flows on a translation surface that are minimal but not uniquely ergodic. We give bounds for the number of generic invariant probability measures.
\end{abstract}
%\tableofcontents
%\newcommand{\mc}[1]{}
%
%       This macros file is mean to be used
%       with "newlatex" or "amslatex".
%       It has the advantage of standardized
%       blackboard and gothic fonts.
%
%       To invoke, use:
%     Multiple  Saddle  Connections  on  Flat  Surfaces  and
%            Principal Boundary of the Moduli Spaces
%                   of Quadratic Differentials
%
%                    by H.Masur and A.Zorich
%
%----------------------------------------------------------------
%
%     File: prin2005.tex
%     LasTheorem
\section{Introduction and Statement of Theorem}

 An  Abelian  differential  or holomorphic $1$-form $\omega$ on a Riemann surface $X$, assigns to each  holomorphic coordinate $z$ on $X$ a holomorphic  function  $f^z(z)$, which in an overlapping coordinate $w$, transforms by 
 $$f^w(w)\frac{dw}{dz}= f^z(z).$$  If the Riemann surface is closed of genus $g\geq 1$ an Abelian differential $\omega$ has zeroes  of order $(\alpha_1,\ldots, \alpha_s)$  with $\sum_{i=1}^s\alpha_i=2g-2$. 
 %Since all quadratic differentials will be assumed to be finite area, the poles are at most simple poles. 
%% We only allow simple poles.
% 
In a more geometric fashion one can also describe $\omega$ as a union of polygons each embedded in $\mathbb{C}$ with pairs of sides identified by translations. Since each polygon is embedded in $\mathbb{C}$, letting $z$ be the local coordinate on the polygon, one defines the holomorphic $1$-form  $\omega$  to be $dz$ in each polygon away from the vertices. The fact that sides are identified by translations says  $dz$ defines a global $1$-form.  Every Abelian  differential $\omega$ can be expressed this way.  This justifies the term {\em translation surface},  and these surfaces are  usually denoted   by $(X,\omega)$. 
The $1$-form $\omega$ defines a metric $|\omega|$ and an area form $|\omega|^2$.  In the polygon description these are just the Euclidean metric and Lebesgue measure. 

 For each direction $0\leq \theta<2\pi$ there is a straight line flow $\phi_\theta^t:(X,\omega)\to(X,\omega)$  in direction $\theta$. If $\omega$ is written locally as $dz$ then these are the lines that make angle $\theta$ with the positive real axis.    Since translations preserve slopes this gives a flow defined for all times except at the points which encounter a zero.   Lebesgue measure is invariant under the flow.      On a flat torus 
the classical Weyl theorem says that the minimality of the flow implies unique ergodicity meaning that Lebesgue measure is the only invariant probability measure for the flow. Equivalently, every orbit is equidistributed on the torus.   In genus $g\geq 2$, however, there are examples of minimal flows on translation surfaces  that are not  uniquely ergodic. The first examples appeared in the papers  of Veech  \cite{V} and Satayev \cite{S}.  In a slightly different context there are  interval exchange  transformations (IET) that are minimal, but not uniquely ergodic. An example  based on the construction of  Veech appeared in \cite {KN}. Then  Keane  \cite{K} gave a general method of constructing examples.   It is a basic fact that the set of invariant   
probability measures forms a convex set and the extreme points are ergodic.   The motivation for this paper is the following classical theorem  of Katok  (\cite{Ka}).
  \begin{thm}
Suppose $(X,\omega)$ is a translation surface on a closed surface of $g\geq 2$.    If $\phi_\theta^t$ is   a minimal straight line flow on $(X,\omega)$, then there are at most  $g$    ergodic  probability measures. 
           \end{thm}

The goal of this paper is to extend the Katok result to bound the number of generic measures. 

\begin{defin}
Let $\phi^t$ be a measure preserving flow on a probability space $(X,\mu)$.  A point $p$ is $2$-sided {\em generic} for $\mu$ if for all continuous $f:X\to \mathbb{R}$,$$\lim_{S,T\to \infty}\frac{1}{S+T}\int_{-S}^Tf(\phi^t(p))dt=\int_X f d\mu.$$ 
\end{defin}

The Birkhoff Theorem says that if an invariant  measure is ergodic then  almost every point is $2$-sided generic for the invariant measure.  A  measure may have  generic points without being ergodic. For example for the full shift on $d\geq 2$ letters every  invariant measure has generic points.

\begin{defin}
A measure $\mu$ is $2$-sided {\em generic} if it has a $2$-sided generic point.
\end{defin}

%\begin{defin}
%Given a vertical flow $\phi^t$ on a translation surface $(X,\omega)$, a generic invariant measure $\mu$ with transverse measure $\rho$, a horizontal interval $I$ and a pair $L>0,\epsilon>0$ we say a point $q$ is an {\em $\epsilon$-effective  point} for $\rho$, $L$,  and $I$,  if any  vertical line $\gamma$ through $q$  with $|\gamma|\geq L$ 
%satisfies \begin{equation}
%\label{eq:tau}
%\left|\text{card}(\gamma\cap I)-\rho(I)|\gamma|\right|<\epsilon|\tau|.
%\end{equation}
%\end{defin}

In the rest of the paper we will  just use the description generic to refer to a $2$-sided measure or $2$-sided generic point.
If a point satisfies $$\lim_{T\to \infty}\frac{1}{T}\int_0^Tf(\phi^t(p))dt=\int_X f d\mu,$$ 
we will say it is a {\em forward} generic point and the  measure is forward generic.  
This is the definition of generic given by V. Cyr and B.Kra in \cite{CK},  where bounds are given for the number of forward generic measures in quite general situations. Their results  will be   discussed in the next section.  
The purpose of this paper is to prove the following result. 
\begin{thm}
 \label{thm:main}
%either a single zero or a pair of zeroes of equal order. 
%Assume $(X,\omega)$ has at most $2$ zeroes.  
Let $(X,\omega)$ be a translation surface on a closed surface of genus $g\geq 2$ with $s$ zeroes.   Let $\phi_\theta^t$ be a minimal straight line flow on $(X,\omega)$ which is not uniquely ergodic. 
Then 
%\item for $g=2$ there are no invariant generic measures that are not ergodic.
 the number of invariant  generic probability measures is bounded by $g+s-1$.
\end{thm}
We  do not know if the bound in Theorem~\ref{thm:main} is in general sharp.
 An  interesting question is if forward generic implies  generic in the case of flows on translation surfaces, or equivalently interval exchange transformations.    We do not in general have a  bound  for the number of  forward generic measures that improves on the bounds found in \cite{CK}. 
There  is one  special case where we can bound the number of forward  generic measures that does improve on their bounds. 
 \begin{thm}
 \label{thm:special} 
  If $(X,\omega)$ has genus $g\geq 2$ with $1$ or  $2$ zeroes,  and there are  $g$  ergodic measures,  then there are no  invariant forward  (hence $2$-sided) generic non-ergodic probability measures. 
%If $(X,\omega)$ has genus $g\geq 3$  with a single zero and $g-1$ ergodic measures then there is at most $1$  forward generic non-ergodic measure.

\end{thm}

An example  is the case of genus $2$ with two zeroes.   If the flow is minimal, but not uniquely ergodic, there are two ergodic measures and no other forward generic measures.  

 \noindent  \textbf{Acknowledgements}
The author  was supported in part by NSF grant DMS-1607512.
The author wishes to thank the referee,    Jonathan  Chaika, Bryna Kra, Alex Eskin, Kasra Rafi, and Alex Wright for helpful  conversations.

\subsection{Connections to Interval Exchange Transformations and History}
\medskip

Recall to define an interval exchange one is  given $x=(x_1,x_2,\ldots,x_d)\in \mathbb{R}^d$ where $x_i > 0$. Form the  $d$ sub-intervals of the
interval $[0,\sum_i x_i)$: $$I_1=[0,x_1) ,
I_2=[x_1,x_1+x_2),\ldots,I_d=[x_1+\ldots x_{d-1}, x_1+\ldots+x_{d-1}+x_d).$$ Given
 a permutation $\pi$ on  the set $\{1,2,\ldots,d\}$, one obtains a d-\emph{Interval Exchange Transformation} (IET)  $ T_{\pi,x\color{black}} \colon [0,\underset{i=1}{\overset{d}{\sum}} x_i) \to
 [0,\underset{i=1}{\overset{d}{\sum}} x_i)$ which exchanges the intervals $I_i$ according to $\pi$. That is, if $z \in I_j$ then $$T_{\pi, x\color{black}}(z\color{black})= z - \underset{k<j}{\sum} x_k +\underset{\pi(k')<\pi(j)}{\sum} x_{k'}.$$
Lebesgue measure is invariant under the action of $T$.

There  is a relationship between straight line flows on translation surfaces $(X,\omega)$ and IET. The first return map to a transversal of the flow is an IET. Conversely, an IET can be suspended  to form a closed translation surface.
  The genus $g$ and number of zeroes $s$  depend not only on the number of intervals  but also on 
 the permutation  $\pi$ of the IET.  For complete details we refer to \cite{Z}.
  For example, if $\pi$ is the hyperelliptic permutation $$\pi (j)=d-j+1$$ for $j=1,\ldots, d-1$ and $d$ is even, 
then the suspended $(X,\omega)$ has genus $g=\frac{d}{2}$ and  a single zero. If $d$ is odd, then the suspended $(X,\omega)$ has genus $g=\frac{d-1}{2}$ and two zeroes. Theorem~\ref{thm:main} gives a bound of $\frac{d}{2}$ and $\frac{d+1}{2}$ in the two cases.  
 In general, we have 
  $$d=2g+s-1.$$   %In the case of $d=4$ the suspended  $(X,\omega)$ has genus $2$ and a single zero; it lies in the stratum $\mathcal{H}(2)$ and for $d=5$ we get that $(X,\omega)$ has genus $2$ with a pair of zeroes. It lies in $\mathcal{H}(1,1)$.  
% Notice for $d=5$ as remarked above Theorem~\ref{thm:special}   gives exactly $2$ generic measures which is an improvement of the bound $3$ in \cite{CK}, and for $g\geq 3$ the bound $g+s-1$ in this paper is an improvement of the bound $2g+s-3$.

For further results on counting ergodic measures under certain conditions see \cite {DF}.

\subsection{History}

As mentioned before 
the earliest examples of minimal but not uniquely ergodic flows on translation surfaces and equivalently interval exchange transformations were given by Veech in  \cite{V}, Sataev in  \cite{S}, and Keane in \cite{K}.

More recent examples of minimal but not uniquely ergodic flows using topological methods  have been given by Gabai in \cite{G} and Lehnzen, Leininger, Rafi  in \cite{LLR}. 
In \cite{CM} Chaika-Masur constructed a minimal  interval exchange transformation on $6$ intervals  with exactly $2$ ergodic measures and one additional $2$-sided generic measure which is not ergodic. The question arose whether there were other  generic  measures, in addition to the one found.  The IET can be suspended to give a genus $3$ translation surface with one zero.  Theorem~\ref{thm:main}  says that there is at most one $2$-sided generic non-ergodic measure  so in fact there are no other generic measures in this
$6$ interval IET example, besides the one in \cite{CM}.  

A desire for a bound on the number of generic measures was also inspired by work \cite{CK} of V. Cyr and B. Kra. They studied subshifts with linear word growth.
They showed that if  a subshift on $d$ letters  satisfies  $$\limsup_{n\to\infty}  \frac{P(n)}{n}<k,$$ then the number of distinct non atomic forward generic measures  is bounded by $k-2$. Here $P(n)$ is the number of words of  length  $n$.  

As an application of their general result 
they showed that for a minimal interval exchange on $d$ letters, the above $\limsup$ is strictly smaller than $d$, so they have  a bound of $d-2$ for the number of forward generic measures.   The bound in \cite{CK} is therefore $d-2=2g+s-3$ forward generic measures for IET on $d$ letters.   Theorem~\ref{thm:main} gives the  smaller bound $g+s-1$ for the number of ($2$-sided) generic measures. But again as we do not know if forward generic implies $2$-sided generic we do not know if this improves their bound in general, but it is an improvement 
 in the special case  covered by 
Theorem~\ref{thm:special}.
%We give application of our main theorem to IET. 
%In minimal stratum there are $2g$ intervals.
 \subsection{Outline}
We briefly outline the idea of the proof of Theorem~\ref{thm:main}. We use Teichm\"uller dynamics and the idea of renormalization.  By rotating, we can assume that the straight line flow is in the vertical direction.  We assume throughout the paper that it is minimal, but not uniquely ergodic. We then apply the Teichm\"uller geodesic flow $g_t$ to the surface $(X,\omega)$.   Along the new surfaces $g_t(X,\omega)$ the  vertical lines of the translation surface flow are contracted, and the horizontal lines expanded. It is known, \cite{M},  that  as $t\to\infty$, the Riemann surfaces along  $g_t(X,\omega)$ eventually leave every compact set  
of the moduli space of Riemann surfaces.  This means that for large $t$    there is a maximal collection of disjoint curves that are hyperbolically short or equivalently short in extremal length. The curves depend on $t$.     The complimentary components are either cylinders or thick surfaces which means that they do not have any short essential curves.  This defines  a thick-thin decomposition of the surface.  We show first in  Proposition ~\ref{prop:ergodic}  that  the image under $g_t$ of generic points on $(X,\omega)$ of different  ergodic measures,  lie in different connected components  of the  complement of the short curves and have area bounded below away from $0$.   These can be either thick surfaces or  cylinders.  Thus we can associate to each ergodic measure and large time, a subsurface depending on the time,  and further that different ergodic measures are associated to {\em disjoint} subsurfaces.    These ideas are related to work of McMullen in \cite{Mc}.  He found a bound for the number of ergodic measures under a certain no loss of mass assumption in terms of the number of connected components. 

 If there are  generic measures that are not ergodic, they will be associated to subsets of  other  complements of the short curves. Because the set of generic points of these measures has measure $0$ it turns out  these subsets may not be an  entire complimentary component  which leads  to complications in the counting of the number of possible measures.     To guarantee that distinct generic measures determine {\em disjoint} subsets we need them to be  $2$-sided except in the cases of Theorem~\ref{thm:special}.  This analysis is carried out in     Proposition~\ref{prop:separating}.   Counting the maximum  number of disjoint subsets of the original surface,   will give the theorem.

 The arguments necessary  for Proposition ~\ref{prop:ergodic}  and Proposition~\ref{prop:separating}  begin with  Proposition~\ref{prop:rectangles}.  It says  that vertical lines through the images of generic points of different  measures cannot bound rectangles on  $g_t(X,\omega)$ for large $t$.  
The rectangle argument is used in Proposition~\ref{prop:ergodic}  
to show that  subsurfaces  that have limits of positive area in the Deligne-Mumford compactification or are cylinders with circumferences going to $0$ and area bounded away from $0$ provide  the desired subsurfaces for the ergodic measures.

   To use limiting ideas in considering non ergodic measures requires  arguments where one renormalizes these complementary  surfaces, to account for the fact that areas approach $0$.   This is carried out in  Proposition~\ref{prop:separating}, based on work in \cite{EKZ}, which in turn used work in \cite{R1} and \cite{R2}. We will also need a preliminary result, Lemma~\ref{lem:effectivegen} on quantitative genericity. 

\subsection{Notation}
Given quantities $x,y$ we use the notation $x\emul y$ to indicate there is a constant $C$ depending only on the genus $g$ so that $$\frac{1}{C}x\leq y\leq Cx.$$

\section{Translation surfaces, Genericity}

\subsection{Saddle connection, cylinders}

 A translation surface $(X,\omega)$   defines a metric  $|\omega(z)dz|$ on $X$ which is flat except at the singularities, which have concentrated negative curvature.  In the polygon version one takes the  Euclidean metric $|dz|$ in each polygon. Translations preserve the metric.  Moreover slopes of lines are preserved  under the side identifications.  In particular this means that each slope defines a flow by lines of that slope.   A line leaving a zero is called a {\em separatrice}. A line segment joining singularities   without singularities in its interior is called a {\em saddle connection}. 
The holomorphic $1$-form $\omega$ defines the area form $|\omega|^2$. We will assume that our translation surfaces have unit area $$\int_X |\omega|^2=1.$$
If the metric $|\omega|$  is understood we will denote the length of a curve by $|\alpha|$
 
  Given an oriented line segment  $\gamma$ one defines the holonomy of $\gamma$ by $\text{hol}(\gamma) := \int_\gamma \omega\in \mathbb{C}$.
The real and imaginary parts of the holonomy are also called the horizontal and vertical components of the holonomy.

 If a geodesic $\beta$ joins a nonsingular point to itself without passing through a singularity it is the core curve of a {\em cylinder} $C(\beta)$,  i.e. the isometric image of a Euclidean cylinder $[0, a] \times (0, b) / (0, y) \sim (a, y)$ for some positive real numbers $a$ and $b$. Then $b$ is the height of the cylinder. We also refer to it as the distance across the cylinder. The cylinder is swept out by closed parallel  loops homotopic to $\beta$. We will suppose throughout that all cylinders are maximal, i.e. in the notation of the previous sentence that $b$ is as large as possible.   The boundary of the cylinder is then composed of saddle connections.  There are also saddle connections crossing the cylinder joining zeroes on opposite boundary components.

 The angle around a zero of order $\alpha_i$ is $2\pi (\alpha_i+1)$ and is called a cone angle. For every  free homotopy class $\gamma$ of a simple closed curve there is a geodesic in its homotopy class. It is either  unique or there is a cylinder. If it  is unique, then it is a union of saddle connections. The angle at any zero between incoming and outgoing saddle connections is at least $\pi$.

  It is a classical result of Strebel's \cite{Str} that on  a closed translation surface (or more generally on any half translation surface or quadratic differential), for  any direction, the flow in that direction decomposes into cylinders and minimal domains in which  every line that does not hit a zero is dense in that domain.  If the minimal domain is not the entire surface then the boundary consists of saddle connections in that direction.
  
  \subsection{Trapezoids}
  We also need to discuss translation surfaces with boundary. A complete treatment can be found in \cite{Str}.   We consider  the horizontal direction. We assume that each boundary component    consists of a union of saddle connections, but assume none are horizontal. 
In the horizontal direction there may still be minimal domains and cylinders. 
There are in addition,   horizontal segments leaving points on each boundary component.      A horizontal segment leaving the boundary at a nonzero point either hits a   zero in the interior  or returns to a (possibly the same) boundary component.  
%If it returns to the same boundary component, it is   homotopically nontrivial, which means that together with an arc  on the boundary it does  not bound a disc.  

If a  segment joins a pair of nonzero points on the two  boundaries, there is a maximal family of parallel horizontal segments joining nonzero points on the same boundary components, determining  a  trapezoid.  The boundary of the trapezoid  has two horizontal sides   each of which contains a zero, either on the boundary of the surface  or in the interior. If the only zero  on  a horizontal side is on the boundary, then at the zero there are at least two horizontal separatrices  leaving the zero and entering the surface.   For if there was only one, the trapezoid could be extended, contrary to the assumption it is maximal.  The other two sides of the trapezoid are subsegments of the boundary components.   On each horizontal side of the trapezoid there is at least one  segment that the trapezoid shares with a different domain, which is another trapezoid, a cylinder or a minimal domain.

\subsection{Transverse Measures and Effective Genericity}

An invariant measure $\mu$ for the vertical flow $\phi^t$ on $(X,\omega)$ defines a {\em transverse} measure $\rho$ on  horizontal segments  $I$ by assigning to each segment  $I$ and sufficiently  small $t_0$, $$\rho(I)=\frac{\mu(I\times [0,t_0])}{t_0},$$ where $I\times [0,t_0]$ is the set of points $\phi^t(x)$ where $x\in I$ and $0\leq t\leq t_0$.   The fact that  $\mu$ is measure preserving under $\phi^t$ implies $\rho(I)$  is independent of $t_0$ for small $t_0$. Furthermore it is easy to see that the transverse measure is flow invariant.   
It also follows that if $\mu$ and $\mu'$ are distinct flow invariant probability measures, then the associated transverse measures $\rho$ and $\rho'$ are distinct as well.

\begin{lem}
\label{lem:effectivegen}
Given  a vertical flow on a translation surface with invariant measure $\mu$ and transverse measure $\rho$,  a pair of numbers $0<\eta<1$  and $0<\epsilon<1$,  a horizontal interval $I$, and a ($2$-sided) generic point $q$ for $\mu$, there exists   $L_0$, such that for the  vertical  line $\gamma$ of length $L>L_0$ in either direction, with one endpoint $q$, any subsegment $\gamma'$ of $\gamma$ of length at least $\eta |\gamma|$ 
satisfies 
$$|card(\gamma'\cap I)-\rho(I)|\gamma'||\leq \frac{2\epsilon|\gamma'|}{\eta}.$$

% is $\frac{2\epsilon}{\eta}$-effective for $I$, $\rho$ and $L$. 

\end{lem}

 %Choose $\epsilon$ so 
 %$$\epsilon<\frac{\eta|\mu_1(I)-\mu_2(I)|}{8}$$  and then $L_0$  
\begin{proof}
Choose  $L_1$ large enough 
so that  if $\tau$ is a vertical segment in either direction with one endpoint $q$ and $|\tau|\geq L_1$, then  $$|\text{card}(\tau\cap I)- |\tau|\rho(I)|<\epsilon|\tau|.$$ Let $M$ be the maximum of the pair of numbers of intersection the  vertical segments of length  $L_1$ in the two directions with one endpoint $q$ have with $I$.  Let \begin{equation}
\label{eq:L}
L_0=\max(\frac{L_1}{\eta}, \frac{\rho(I)L_1+M}{\epsilon}).
\end{equation}
 Now take a segment $\gamma$ of length $L\geq L_0$ and a subsegment $\gamma'\subset\gamma$ of length at least $\eta L\geq L_1$ which starts at some $y$ and ends at some $z$ further away from $q$ than $y$.   Let $\tau$ the subsegment of $\gamma$ starting at $q$ and ending at $z$.  We have $|\tau|\geq L_1$ so
 \begin{equation}
\label{eq:tau}
\left|\text{card}(\tau\cap I)-\rho(I)|\tau|\right|<\epsilon|\tau|.
\end{equation}

%(\ref{eq:tau}) holds.
%$$\left|\text{card}(\tau\cap I)-\rho(I)|\tau|\right|\leq\epsilon |\tau|.$$

Let $\sigma$ be the segment from $q$ to $y$. Then
\begin{equation}
\label{eq:count}
\text{card} (\gamma'\cap I) =\text{card}(\tau\cap I)-\text{card} (\sigma\cap I)
\end{equation}

%The first term is $\nu(I)|\gamma''|$ with an error of $\epsilon |\gamma''|$.

The first case is 
if $|\sigma|\geq L_1$.
%Then the second term in (\ref{eq:count}) is $\rho(I)|\sigma|$ with an error of $\epsilon|\sigma|$. 
Then (\ref{eq:tau}) and   (\ref{eq:count}) give $$|\text{card} (\gamma'\cap I)-\rho(I)|\gamma'||\leq \epsilon(|\tau|+|\sigma|)\leq 2\epsilon |\tau|\leq 2\epsilon L=\frac{2\epsilon}{\eta} \eta L\leq \frac{2\epsilon|\gamma'|}{\eta}.$$
We are done in this case.

Now suppose $|\sigma|\leq L_1$ and so $|\gamma'|\leq |\tau|\leq |\gamma'|+L_1$.  Now the second term on the right  in (\ref{eq:count}) is bounded by $M$. 
%and $$|\gamma'|\leq |\tau|\leq L_1+|\gamma'|.$$
Then $$ |\gamma'|\rho(I)-\epsilon (|\gamma'|+L_1)-M\leq \text{card}(\tau\cap I)-M\leq \text{card}(\gamma'\cap I)\leq \text{card}(\tau\cap I)\leq (|\gamma'|+L_1|)\rho(I)+\epsilon (|\gamma'|+L_1).$$ Together with the choice of $L_0$ in  (\ref{eq:L}) this gives   
 $$|\text{card } (\gamma'\cap I)-|\gamma'|\rho(I))|\leq \rho(I)L_1+M+\epsilon(L_1+|\gamma'|)\leq 2\epsilon L\leq\frac{2\epsilon}{\eta} |\gamma'|.$$

\end{proof}

  \section{ Teichm\"uller geodesic flow,  thick-thin decomposition}

The translation surfaces of genus $g$ with fixed
 orders of zeroes $\alpha_1,\ldots, \alpha_s$ with $\sum_{i=1}^s\alpha_i=2g-2$ 
fit together to form a moduli space or stratum $\mathcal{H}(\alpha_1,\ldots,\alpha_s)$.
There is a natural map $$\pi:\mathcal{H}(\alpha_1,\ldots, \alpha_s)\to \mathcal{M}_g,$$ the moduli space of Riemann surfaces which simply records the Riemann surface of $(X,\omega)$. 

There is defined on $\mathcal{H}(\alpha_1,\ldots, \alpha_s)$ an action of the group $SL(2,{\bf R})$.  The action is the linear action on polygons.  The diagonal subgroup $$g_t=\begin{pmatrix} e^{t/2}& 0\\
0& e^{-t/2}
\end{pmatrix}$$ is the  Teichm\"uller  geodesic flow.   It contracts the vertical lines of $\omega$ by $e^{t/2}$ and expands along the horizontal lines by $e^{t/2}$.  We will think of $g_t$ in two ways.  If we represent  translation surfaces as  polygons, then it is a  linear map, called the {\em Teichm\"uller map}  from one  polygon to another.   We can also think of it as a self map from the stratum to itself which takes the point corresponding to one polygon to the point corresponding to its image under the Teichm\"uller map.

We will denote the underlying Riemann surface of $g_t(X,\omega)$ by $X(t)$ and the holomorphic $1$-form by $\omega(t)$.

An excellent short survey of translation surfaces and the $SL(2,{\bf R})$  action on their moduli spaces can be found in \cite{W}.

For the homotopy class $\alpha$ of a  closed curve  on a closed  Riemann surface $X$ the extremal length of $\alpha$ on $X$ is defined to  be  $$\text{Ext}_X(\alpha)=\sup_\sigma \frac{\ell_\sigma^2(\alpha)}{\text{area}(\sigma)}.$$
Here, $\sigma$  ranges over all metrics in the conformal class of $X$  and 
$\ell_\sigma(\alpha)$  is the infimum of the $\sigma$-length of all representatives of the homotopy class of the curve $\alpha$.  

For constants $\epsilon_0>\epsilon_1 > 0$, the $(\epsilon_0,\epsilon_1)$ thick-thin decomposition of $(X,\omega)$ is the pair $(\mathcal{A},\mathcal{Y})$ where $\mathcal{A}$ is the set of geodesic representatives in the metric defined by $(X,\omega)$ of closed curves $\alpha$  such that $\text{Ext}_X(\alpha)\leq\epsilon_0$ and $\mathcal{Y}$ is the set of the components of $X$  cut along $\mathcal{A}$. We assume that $\mathcal{A}$ is maximal in the sense that any curve not in  $\mathcal{A}$ has extremal length at least $\epsilon_1$.  

\begin{rem}
It may happen that distinct short curves in $\mathcal{A}$ share a zero. This would be   the case, for example, if $\omega$ has a single zero and there is more than one short curve. If this happens then the translation surface representative of a thick subsurface has strictly fewer boundary components than the number of short curves on its boundary. 
\end{rem}

We have the following definition of size due to Rafi \cite{R1}.
\begin{defin}
Let $Y$ be a subsurface  of a translation surface $(X,\omega)$ whose boundary consists of a collection of curves with extremal length at most $\epsilon_0$.   If $Y$  is not topologically a $3$-times punctured 
sphere or a cylinder,  the {\em size} $\lambda(Y)$ is defined to be the infimum of the $|\omega|$ lengths of the  essential closed geodesics in $Y$.
A $3$-times punctured sphere  does not have essential curves so the size is defined to be the diameter. If $Y$ is a cylinder $C(\alpha)$,  the   {\em size}  $\lambda(C(\alpha))$ is the length of a vertical segment joining the two boundary components. 

\end{defin}
For the definition in the case of cylinders to make sense it is necessary that  the core curve of the cylinder not to be vertical as it will not be for the rest of the paper since the vertical flow is assumed to be minimal

Let  $(\mathcal{A},\mathcal{Y})$ be the thick-thin decomposition of $(X,\omega)$.    If  $\alpha\in \mathcal{A}$ has a unique geodesic representative as a union of  saddle connections it is  on the boundary of a pair of surfaces  $Y$ and $Z$, which may coincide.  If  there is  a cylinder $C(\alpha)$ then each boundary component of $C(\alpha)$ is on the boundary of such a $Y$.  We will need the following result of  Rafi's which is (part of) Theorem 3.1 of  \cite{R2}. 
\begin{thm}
\label{thm:Rafi}
% Let $(X,\omega)$ be a translation surface and let 
% For $\alpha\in \mathcal{A}$ such that $\alpha$ is on the boundary of $Y$ and also on the boundary of a possibly distinct $Z$ where  $Y,Z\in\mathcal{Y}$,    \begin{equation}

\begin{equation}
\label{eq:modulus}
\frac{1}{\text{Ext}_X(\alpha)}\emul \log \frac{\lambda(Y)}{|\alpha|}+\log \frac{\lambda(Z)}{|\alpha|}+\frac{\lambda(C(\alpha))}{|\alpha|}
\end{equation}
\end{thm}
By way of further explanation, the first two terms are  the {\em moduli} of what is called the {\em expanding annulus} in $Y$ (resp. $Z$) that are  isotopic to $\alpha$. The expanding annulus is foliated by closed curves that are equidistant from the loop $\alpha$. 
The last term on the right is the modulus $\text{Mod}(C(\alpha))$ of the cylinder.  This is the ratio of the height to the circumference of the cylinder.

We also will use Theorem 4 of \cite{R1} which says  
\begin{equation}
\label{eq:boundedbelow}
\text{diam}(Y)\emul   \lambda(Y). 
\end{equation}.

Now suppose  $t_n\to \infty$ is a sequence with the property that for some maximal collection of curves  $\alpha(t_n)$ we have $\text{Ext}(\alpha(t_n))\to 0$ on the Riemann surface $X(t_n)$ of $g_{t_n}(X,\omega)$.  This is equivalent to saying that the hyperbolic lengths of $\alpha(t_n)$ go to zero as well.  In other words we are leaving every compact set in the moduli space $\mathcal{M}_g$.  
Consider any sequence of   components $Y(t_n)$ of the complement of the flat geodesic representatives of $\alpha(t_n)$ along $X(t_n)$.
Assume they are not cylinders.   We write $|\alpha(t_n)|$ to be the flat length of $\alpha(t_n)$ with respect to metric defined by $\omega(t_n)$.
%Let $Y(t_n)$ a thick component of $g_{t_n}(X,\omega)\setminus \cup_{j=1}^k \Omega_{t_n}^j$, It has area going to $0$ as $t_n\to\infty$. 

\begin{defin}
We say  such a sequence of noncylindrical complementary components $Y(t_n)$
%\subset g_{t_n}(X,\omega)\setminus \cup_{j=1}^k\Omega_{t_n}^j$ 
%with $\lim_{t_n\to 0}\text{area} (Y(t_n))  =0$
is a {\em plump} sequence if for all boundary components $\alpha(t_n)$ of $Y(t_n)$,  $$\lim_{t_n\to\infty}\frac{|\alpha(t_n)|}{\lambda(Y(t_n))}=0.$$ 
We say it is   {\em gaunt} if there is some boundary component $\alpha(t_n)$ of $Y(t_n)$ with $$\liminf_{t_n\to\infty}  \frac{|\alpha(t_n)|}{\lambda(Y(t_n))}>0.$$

\end{defin}
%he following lemma is an immediate Corollary of  the above  Theorem~\ref{thm:Rafi} 
We note that if a sequence is gaunt we can pass to a subsequence so that the above $\liminf$ is replaced by limit. 
 It may happen that   $Y(t_n)$  has  empty  interior,  (see \cite{R1} for an example). 
We also may have that $Y(t_n)$ is a disc.
This can happen if the short extremal  length curves bound a sphere, but their representatives in the flat metric form a connected set.

\section{ Limits of translation surfaces}
A  compactification of a stratum of Abelian differentials and its relation to the Deligne-Mumford compactification of the Riemann moduli space is carried out in the paper \cite{BCGGM}. We will not need their entire spectrum of results, but essentially we will use Theorem 10 of \cite{EKZ} which is a theorem about limits of translation surfaces.   We begin with a brief explanation of some ideas from the Deligne Mumford compactification $\overline{\mathcal{M}_g}$ of the moduli space $\mathcal{M}_g$ of  closed Riemann surfaces of genus $g$.  References are  the  papers \cite{HK}, \cite{A}, \cite{Mc}.The points in  
 $\overline{\mathcal{M}_g}\setminus \mathcal{M}_g$ are  possibly disconnected Riemann surfaces $X(\infty)$ with punctures, or in the language of algebraic geometry noded stable algebraic curves. We denote the nodes by $\Sigma(\infty)$. These are points where two distinct surfaces are glued together or two points of the same surface are identified.  By separating the pair of points, we will think of them as {\em punctures}. 

We put a topology on $\overline{\mathcal{M}_g}$ by saying $X_n\to X(\infty)$ if there is an exhaustion of $X(\infty)\setminus \Sigma(\infty)$  by compact sets $K_n$ and for any $\epsilon>0$, for large $n$, injective $(1+\epsilon)$ quasiconformal maps  $h_n:K_n\to X_n$.

This allows us to define if a sequence $p_n\in X_n$ has a limit  $q_\infty\in X(\infty)\setminus\Sigma(\infty)$. Namely,  there is a sequence $q_n\in K_n$ such that   $p_n=h_n(q_n)$ and $q_n\to q_\infty$.

We can extend all this to subsurfaces $X^b(\infty)\subset X(\infty)$ with boundary and say a sequence $X_n^b\subset X_n$ with boundary converges to $X^b(\infty)$ if there is a nested sequence of neighborhoods $U_n$ of $\partial X^b(\infty)$ converging to  the boundary and for any $\epsilon$ there are  $1+\epsilon$ quasiconformal maps of the complement of $U_n$ to $X_n^b$.

We will apply this to 
sequences   $X_j(t_n)$ that are  thick components of thick-thin decomposition of  $\omega(t_n)$
whose boundary satisfies $Ext(\alpha(t_n))\to 0$. 
 %A subsurface  $X_j(\infty)$ of $X(\infty)$  is the limit of $X_j(t_n)$, if  for any sufficiently small neighborhood  $U$ of the punctures $\Sigma_j$ of $X_j(\infty)$,  for large $n$, $h_{t_n}(X_j(\infty)\setminus U)\subset X_j(t_n)$ and  $h_{t_n}(\partial U)$ is  homotopic to $\alpha(t_n)$. 
 
Now let $\omega(t_n)$ a  holomorphic $1$-form on $X(t_n)$.
  Let $\omega_j(t_n)$ the   restriction of $\omega(t_n)$ to $X_j(t_n)$.  
  %We  say   $\omega_j(\infty)$ a  $1$-form on a component $X_j(\infty)$ of $X(\infty)$  is the limit of $\omega_j(t_n)$ as $t_n\to\infty$ if  for every compact $K\subset X(\infty)\setminus \Sigma_j(\infty)$,
 %$$\lim_{t_n\to\infty}(h_{t_n}\vert_K)^*\omega_j(t_n)=\omega_j(\infty).$$ 

\begin{prop}
\label{prop:limit}
Suppose $X_j(t_n)$ is sequence of thick components of the thick-thin decomposition.  By passing to  a subsequence we can assume that there is a limiting surface $X_j^b(\infty)$, possibly with boundary, together with a $1$-form $\omega_j(\infty)$, such that .$\frac{\omega_j(t_n)}{\lambda (X_j(t_n))}$ converges to $\omega_j(\infty)$.  Moreover for boundary loops 
 $\alpha(t_n)$ of $X_j(t_n)$  
 \begin{itemize}
 \item if  $\lim_{t_n\to\infty} \frac{|\alpha(t_n)|}{\lambda (X_j(t_n))}>0$ then $X_j(\infty)$ has a boundary 
 loop  corresponding to $\alpha(t_n)$.  
\item  if  $\frac{|\alpha(t_n)|}{\lambda (X_j(t_n))}\to 0$ then   $X_j(\infty)$ has a puncture corresponding to $\alpha(t_n)$ 
%\item The area of a neighborhood of $\Sigma_j(\infty)$ is finite  with respect to the area form $|\omega(\infty)|^2$.  
\end{itemize}
\end{prop}

\begin{rem}
This is an example of what is called multi-scaling in \cite{BCGGM}. In order to achieve a limit, the $1$-forms need to be scaled by their size, which can depend on the component.
\end{rem}

\begin{proof} Theorem 10 of \cite{EKZ} essentially gives the convergence.  (We remark that in that paper if the assumption of the first bullet holds, one  acquires a pole if one takes  the  limit of the renormalized differential  on a region in $X(t_n)$ that includes $\alpha(t_n)$ in its interior.    In that case the curve $\alpha(t_n)$ itself limits on a geodesic in the homotopy class of a loop surrounding the puncture).  Here we just take a limiton the region $X_j(t_n)$ itself.   

We prove the second bullet. Since the length on the renormalized differential still goes to $0$, the surface must develop a puncture denoted $p$.  Arguing by contradiction, suppose though $\omega_j(\infty)$ has a pole at $p$. Then there is a positive lower bound $a_0$ for the length of any loop surrounding $p$. %On the other hand the assumption says that the flat length $a_n=\int_{\alpha(t_n)}\frac{|\omega_j(t_n)|}{\lambda (X_j(t_n))}$  of the boundary curve $\alpha(t_n)$ has limit  $0$.    
Fix a point $z_0\in X_j(\infty)$ with $z_0\neq p$.
The diameter of $X_j(t_n)$ measured with respect to the normalized metric  $\frac{|\omega_j(t_n)|}{\lambda (X_j(t_n))}$ is uniformly bounded; say by a constant $C$.  Since the distance from $z_0$   to $p$ measured in the metric induced by $|\omega_j(\infty)|$ is infinite,
choose  $U$ a neighborhood of $p$ such that the distance from $z_0$ to $U$   is greater than $2C$.   For large $t_n$, $h_{t_n}(z_0)\in X_j(t_n)$,  where     $h_{t_n}:X_j(\infty)\setminus U\to X(t_n)$ is the quasiconformal map defining the topology.    Since  $\frac{\omega_j(t_n)}{\lambda (X_j(t_n))}$ converges to some $\omega_j(\infty)$ on $X_j(\infty)\setminus U$,
%Since \begin{equation}
%\label{eq:convergence}\frac{1}{\lambda(X_j(t_n))}h_{t_n}^* \omega_j(t_n)\to\omega(\infty).\end{equation}   
we  cannot have $\alpha_{t_n}\subset  h_{t_n}(X_j(\infty)\setminus U)$.  This follows from the assumption  $\frac{|\alpha(t_n)|}{\lambda (X_j(t_n))}\to 0$,
while the  length of $h_{t_n}^{-1}(\alpha(t_n))$ with respect to $|\omega_j(\infty)|$ is bounded below by $a_0>0$.

    Thus $h_{t_n}(\partial U)$ contains points of  $X_j(t_n)$.  But the $|\omega_j(\infty)|$ distance between $z_0$ and $\partial U$ is at least $2C$ while the $\frac{|\omega_j(t_n)|}{\lambda (X_j(t_n))|}$ distance between their images is at most $C$. This is a contradiction.  \end{proof}

\begin{rem}
\label{rem:gauntlimit}

As a consequence of Proposition~\ref{prop:limit}
if the sequence is plump,  the limiting  $\omega(\infty)$ is a finite area holomorphic $1$-form on $Y(\infty)$.  In particular we still have vertical and horizontal line flows defined by $\omega(\infty)$. 
  If the sequence is gaunt,  then for  a (possibly empty) sequences of boundary curves $\alpha(t_n)$ satisfying   $\frac{|\alpha(t_n)|}{\lambda( X_j(t_n))}\to 0$, we still have punctures in the limit,  where the $1$-form is holomorphic. However  by definition there are boundary curves with $\frac{|\alpha(t_n)|}{\lambda(X_j(t_n))}$   bounded away from $0$. (Again note lengths $|\alpha(t_n)|$ are measured with respect to the $1$-form $\omega(t_n)$). In the limit we get a surface with boundary  corresponding to these curves.   
  % If one did that then one could take the geodesic in the class of a loop surrounding the puncture. The corresponding surface with boundary would be exactly as above. 
\end{rem}

 We adopt the notation that if a sequence of loops  $\alpha_j(t_n)$ on the boundary of a sequence of gaunt $Y_j(t_n)$ satisfy
$\frac{|\alpha_j(t_n)|}{\lambda(Y_j(t_n))}\to 0$, we will say it is a {\em short} sequence. Otherwise it is {\em long}.

We have the following which was essentially proved by Eskin, Mirzakhani, and Rafi in \cite{EMR}. 
\begin{lem}
\label{lem:common}
A  gaunt  sequence $Y_j(t_n)$ cannot have  two or more short sequences of boundary loops sharing a common  zero.
\end{lem}
\begin{proof}
 We consider the sequence of loops $\alpha(t_n)$ surrounding the set of all  short boundary curves $\beta(t_n)$  sharing a zero.   The flat length of $\alpha(t_n)$  goes to $0$ on the normalized surface. It cannot be the case that the  extremal length of $\alpha(t_n)$ goes to $0$ for then $Y_j(t_n)$ would be a disc and the surrounded boundary loops  $\beta(t_n)$ would not all be short.  Then suppose the extremal length of $\alpha(t_n)$ is bounded away from $0$.   Let  $\gamma(t_n)$ be any closed curve on $Y_j(t_n)$ with bounded flat length.  Since the flat length of $\alpha(t_n)\to 0$,  by the exact same proof as given in Lemma 3.9 of \cite{EMR} there is a curve $\tau(t_n)\neq \alpha(t_n)$ separating $\gamma(t_n)$ from $\alpha(t_n)$ hence from the $\beta(t_n)$ with the  extremal length of $\tau(t_n)$ going to $0$. But then the boundary loops $\beta(t_n)$  are not on the boundary of $Y_j(t_n)$,  again  a contradiction. 
\end{proof}

\begin{defin}
Suppose $Y_j(t_n)$  is gaunt sequence.  We define  a {\em short boundary segment} $J(t_n)$ to be a connected segment on a short boundary loop  of  $Y_j(t_n)$ each of whose  endpoints is   either  a zero  or  a  point lying on a vertical segment that hits a zero. We require that there be no other such points in the interior of $J(t_n)$.
\end{defin} 

It is clear from the definition that short boundary segments are in {\em pairs}; every vertical segment leaving one hits the other.

\section{Main Propositions}
This section is devoted to 
 Proposition~\ref{prop:ergodic} and  Proposition~\ref{prop:separating}, the main ingredients in the proof of the main theorem. Their proofs  are based on ideas from 
  \cite{M} and  \cite{Mc}. We are supposing the vertical line flow $\phi^t$ is minimal, but not uniquely ergodic.  In \cite{M} it was shown that   as  $t\to\infty$, the Riemann surface $X(t)$ of $g_t(X,\omega)$ eventually leaves every compact set in the moduli space $\mathcal{M}_g$. As discussed previously this means that for all  large 
times $t$ there is a disjoint collection of simple closed curves $\alpha(t)$ whose extremal  or hyperbolic lengths approach $0$.  They define a {\em thick-thin} decomposition of $X(t)$.   The curves $\alpha(t)$ have representatives as geodesics in the flat metric defined by $\omega(t)$ on $X(t)$. Pulled back  by $g_t$ their  representatives as geodesics on $X$ have horizontal component of their holonomy going to zero and  since the flow $\phi^t$ is minimal, the vertical component of their holonomy   on $X$ goes  to infinity, so by the definition of the Teichm\"uller flow, their vertical  holonomy, hence lengths, denoted  $|\alpha(t)|$ on $X(t)$ satisfy 
 \begin{equation}
\label{eq:lengths}
\lim_{t\to\infty}e^{t/2}|\alpha(t)|=\infty.
\end{equation}

\begin{defin}
 Suppose   $\beta,\beta'$  are vertical segments of the same length $|\beta|=|\beta'|$ on a translation surface $(Y,\omega)$.  Given $0<\eta\leq 1$ we say that $\beta$ and $\beta'$   {\em $\eta$-interact}, if there are   subsegments of $\beta$ and $\beta'$ of length at least $\eta |\beta|$ that are the two vertical sides of an isometrically embedded rectangle on $(Y,\omega)$. 
 \end{defin}

In what follows we have a sequence of times $t_n\to\infty$ and consider the $1$ forms $\omega(t_n)$ of $g_{t_n}(X,\omega)$. Again we will think of $g_{t_n}$ as maps of the Riemann surface $X$ to a Riemann surface $X(t_n)$  
We will denote by $\omega_\ell(t_n)$ the  restriction of the $1$-form to a subsurface $X_\ell(t_n)$ of the thick-thin decomposition.

%We will denote generic measures whether ergodic or not by $\rho$. 

\begin{prop}
\label{prop:rectangles}
Suppose  $p,p'$ are generic points of distinct ($2$-sided) generic measures $\mu,\mu'$. 
%For every   $\eta>0$ there exists $L_0$ such that if 
Suppose 
  $\beta(t_n)$ and $\beta'(t_n)$ are vertical lines with endpoints  $g_{t_n}(p)$ and $g_{t_n}(p')$  of $g_{t_n}(X,\omega)$ of equal length  such that as $t_n\to\infty$, $$e^{t_n/2}|\beta(t_n)|\to\infty.$$
  %\at least $\ell_t:=e^{-t}L_0$ then they do not and $\hat\beta_{t_n},\hat\beta_{t_n}'$ are subsegments  of equal length $\eta |\beta_{t_n}|$, 
  Suppose they $\eta(t_n)$ interact. Then $\lim_{t_n\to\infty} \eta(t_n)=0$.   \end{prop} 
%$\hat\beta_{t_n}$ and  $\hat\beta_{t_n}'$ are not two sides of an embedded  rectangle on $g_{t_n}(X,\omega)$. 

\begin{proof}
Suppose on the contrary for a subsequence of $t_n\to\infty$ they $\eta$-interact for some fixed $\eta>0$. That is, there exists subsegments 
$\hat\beta(t_n)\subset\beta(t_n)$ and $\hat\beta'(t_n)\subset\beta'(t_n)$ of equal length $\eta|\beta(t_n)|$ which are two vertical sides of a rectangle on $X(t_n)$.   Let $I$ be a horizontal interval with respect to the $1$-form on the base surface $(X,\omega)$ such that the transverse measures $\rho,\rho'$ corresponding to $\mu,\mu'$  satisfy $\rho(I)\neq \rho'(I)$. Let  $$\epsilon=\frac{\eta|\rho(I)-\rho'(I)|}{8}.$$  Then for $\eta, \epsilon$ and interval $I$, choose  $L_0$ large enough so that it satisfies  the conclusion of Lemma~\ref{lem:effectivegen} and also so  that $$L_0\geq \frac{2}{\epsilon}.$$ 
 %,\beta'_{t_n}$ of equal length at least $L_0e^{-t_n}$. 
%Let $\beta=g_{-t_n}(\beta_{t_n})$ and $\beta'=g_{t_n}^{-1}(\beta_{t_n}')$. We have $|\beta|=|\beta'|\geq L_0$. 
Let $\hat\beta=g_{t_n}^{-1}(\hat \beta(t_n))$ and $\hat\beta'=g_{t_n}^{-1}(\hat \beta'(t_n))$ the equal length vertical segments pulled back to  $(X,\omega)$ by the Teichm\"uller map $g_{t_n}^{-1}$.
 %They have generic points $p$ and $p'$ respectively as endpoints. 
 Then  $$|\hat \beta'|_\omega=|\hat\beta|_\omega=|g_{t_n}^{-1}(\hat \beta(t_n))|_\omega=\eta e^{t_n/2}|\beta(t_n)|\geq \eta L_0,$$ the last inequality holding for large enough $t_n$.
(The lengths on the left are measured with respect to the flat metric on the base surface $(X,\omega)$).

Applying 
Lemma~\ref{lem:effectivegen} to the subsegments $\hat\beta$ and $\hat\beta'$ and the interval $I$,
we get both 
$$|\text{card}(\hat\beta\cap I)-\rho(I)|\hat\beta|_\omega|\leq \frac{2\epsilon}{\eta}|\hat\beta|_\omega,\text{and}\
|\text{card}(\hat\beta'\cap I)-\rho'(I)|\hat\beta'|_\omega|\leq \frac{2\epsilon}{\eta}|\hat\beta'|_\omega.$$ The triangle inequality then gives

$$|\text{card}(\hat\beta\cap I)-\text{card}(\hat\beta'\cap I)|\geq 
|\hat\beta|_\omega|\rho(I)-\rho'(I)| -\frac{4\epsilon}{\eta} |\hat\beta|_\omega=\frac{8\epsilon}{\eta}|\hat \beta|_\omega-\frac{4\epsilon}{\eta} |\hat\beta|_\omega\geq 4\epsilon L_0\geq 2.$$

  Applying  $g_{t_n}$
we see  that the horizontal segment $g_{t_n}(I)$ satisfies  
$$|\text{card}(g_{t_n}(I)\cap \hat\beta(t_n))-\text{card}(g_{t_n}(I)\cap \hat\beta'(t_n))|\geq 2,$$ which means $\hat\beta(t_n),\hat\beta'(t_n)$  cannot be two vertical sides of an embedded  rectangle.

\end{proof}

%In the next proposition $\text{Hyp}(\gamma)$ refers to the hyperbolic length on the underlying Riemann surface of the translation surface.  

Now suppose  $\mu_1,\ldots, \mu_k$ are the ergodic probability measures and $\nu_1,\ldots, \nu_m$ are generic but not ergodic probability measures for the minimal vertical flow $\phi^t$ of $(X,\omega)$.  We have $k\geq 2$. Our goal in proving Theorem~\ref{thm:main}  is to show $$k+m\leq g+s-1.$$

%Fix  generic points $q_1,\ldots q_l$ for the generic but not ergodic measures. 

Let $\mu$ be Lebesgue measure on $(X,\omega)$. 
Since the ergodic measures are the extreme points of the simplex of invariant measures,  there exists $a_i\geq 0$ such that  $$\mu=\sum_{i=1}^k  a_i\mu_i.$$

 For each $i$ the set of generic points of $\mu_i$ 
has Lebesgue measure $a_i$. 
By choosing a measure in the interior of the simplex, and taking the vertical line flow with that measure (on a possibly different $(X,\omega))$  we can assume $a_i>0$ for all $i$.  Let $$A_0=\min_i a_i.$$

\begin{prop}
\label{prop:ergodic}
%Suppose the vertical flow $\phi^t$ on $(X,\omega)$ is minimal but not uniquely ergodic.  Then 
  For any sequence of times $t_n\to\infty$ there is a subsequence, again denoted $t_n$,   an integer $\ell\geq k$,  a number $a_0>0$ and  disjoint open subsurfaces  or cylinders  $X_1(t_n), \ldots, X_\ell(t_n)\subset g_{t_n}(X,\omega)$ 
 such that  %each of which carries the $1$-form $\omega^j(t_n)$, which is the restriction of $\omega_{t_n}$ to $X^j_{t_n}$  
  \begin{enumerate}
\item the boundary of each $X_i(t_n)$ is made up of 
   a union of loops  $\alpha_i(t_n)$ each of which in turn is a union of saddle connections and such that
 $\lim_{t_n\to\infty}\text{Ext}(\alpha_i(t_n))=0$.  
%\item  $\lim_{t_n\to\infty}\sum_{\ell=1}^{k'}\int_{X_\ell(t_n)}|\omega_\ell(t_n)|^2 =1=\int_{X(t_n)}|\omega(t_n)|^2$
\item for each $i$, $\text{area} (X_i(t_n))\geq a_0$ 
\item   If $X_i(t_n)$ are not cylinders  and $\omega_i(t_n)$ is the restriction of $\omega(t_n)$ to $X_i(t_n)$ then $\lim_{t_n\to\infty}\omega_i(t_n)$ exists and is a non-zero finite area $1$-form $\omega_i(\infty)$  on a surface $X_i(\infty)$ in the Deligne-Mumford compactification,
\item  For  each  sequence  $X_i(t_n)$ there  is an ergodic  measure $\mu_j$,   a generic point $p_j$ of  $\mu_j$, independent of $t_n$ such that   $g_{t_n}(p_j)\in 
 X_i(t_n)$. In addition,  if the $X_i(t_n)$ are not cylinders, then $g_{t_n}(p_j)$ has a limit in a compact subset of $X_i(\infty)$.
%\item  Conversely to (3), for each ergodic measure $\mu_j$, there is a generic point $p_j$ of $\mu_j$, and a sequence $X_\ell(t_n)$, such that for sufficiently large $t_n$,  $g_{t_n}(p_j)\in X_{\ell}(t_n)$.
\item If  $X_i(t_n)$ and $\mu_j$ are related as in (4)  and $q$ is a generic point for a measure $\nu\neq \mu_j$, then  for large $t_n$, $g_{t_n}(q)
\notin \overline{X_i(t_n)}$.
\item $\lim_{t_n\to\infty} \text{area}( g_{t_n}(X,\omega)\setminus \cup_i X_i(t_n))\to 0$.

\end{enumerate}
\end{prop}

%Here $|\omega_j(t_n)|^2$ denotes the area form.

\begin{defin}
We say that  $X_i(t_n)$ given in (4) is  {\em associated} to the generic point $p_j$ and generic measure $\mu_j$
\end{defin}

%\begin{rem} An ergodic measure $\mu$ may be associated to two or more $X_\ell(t_n)$ and distinct generic points. 
%\end{rem}

\begin{rem}
Some of the ideas in this   proposition are  contained in  Theorem 1.4 of  \cite{Mc}. One difference is  that in that theorem there is  the assumption that there is no loss of mass in passing to limits of the $1$-form on surfaces in  the Deligne-Mumford compactification.  Here we  allow for the possibility of cylinders with area   bounded below.
\end{rem}
\begin{proof}
By passing to a subsequence we can assume that the Riemann surfaces $X(t_n)$ converge in the Deligne-Mumford compactification to a  possibly disconnected Riemann surface $X(\infty)$  with  punctures.  A collection of curves have their extremal length approaching $0$ so (1) holds.  We consider sequences of those components $X_i(t_n)$ of the $(\epsilon_0,\epsilon_1)$ thick-thin decomposition of  $g_{t_n}(X,\omega)$ whose area determined  by the restriction $\omega_i(t_n)$   of the $1$-form $\omega(t_n)$ to $X_i(t_n)$ is bounded away from $0$. For these   (2) holds.  Among those that are not cylinders,  
 the area $\int_{X_i(t_n)}|\omega_i(t_n)|^2$ is bounded above by the size $\lambda(X_\ell(t_n))$ up to a uniform multiplicative constant. This follows from the fact, Lemma 4.2 of \cite{EKZ}, that one can triangulate $X_i(t_n)$  by saddle connections that are comparable in length to the size.  Thus the size $\lambda(X_i(t_n ))$ is bounded below away from $0$.  Theorem 10 of \cite{EKZ} says that by  passing to subsequences 
%the  {\em normalized} $\frac{\omega(t_n)}{\lambda(t_n)}$ 
$\omega_i(t_n)$ converges to a nonzero finite area holomorphic  $1$-form $\omega_i(\infty)$ on the limiting Riemann surface $X_i(\infty)$.
The limiting $1$-form does not have poles at the punctures. This proves (3). 
 
 Now consider any such sequence $X_i(t_n)$ which has limit $X_i(\infty)$  which has positive area. 
 Since the union of the sets of generic points for the ergodic measures  has full Lebesgue measure, it follows that  any open set $U\subset X_i(\infty)\setminus \Sigma$,  contains limit points of $g_{t_n}(p_j)$ where $p_j$ is generic for some ergodic measure $\mu_j$. (This argument appears in Corollary A.3 in the Appendix of \cite{Mc}) and gives (4) in the non cylinder case.
 
 We next consider the  case that there are components $X_i(t_n)$ of the $(\epsilon_0,\epsilon_1)$  thick-thin decomposition that form  a sequence of cylinders  with areas bounded below away from $0$.   Since the areas are bounded below and the set of generic points for the ergodic measures  has full measure, it follows that by passing to a subsequence that there is a generic point $p_j$ for an ergodic $\mu_j$ such that $g_{t_n}(p_j)\in X_i(t_n)$ for all $t_n$.  This gives (4) in both cases.
 
%Since in (3) we are considering all domains with area bounded away from $0$,  (4) holds automatically.
%Since the measure of the set of generic points of each ergodic measure $\mu_i$ is at least $A_0$,  and the complement of the union of the $X_\ell(t_n)$ has area going to zero,    by passing to a further subsequence, for large $t_n$, there must be a generic point  in one of $X_\ell(t_n)$, proving (4).

 We now verify (5). In the case of  a limiting component $X_i(\infty)$ in the Deligne-Mumford compactification, 
  we note  that  limits points of  $g_{t_n}(q)$ and $g_{t_n}(p_j)$ cannot lie  in the same isometrically embedded open disc of $X_i(\infty)$.   For if they did, then there would be vertical segments $\gamma$ and $\gamma_i$ of $\omega_i(\infty)$,  of the same length each containing a limit point and which are two sides of a rectangle $R(\infty)$.   The pair $\gamma$ and $\gamma_i$ are limits of vertical segments $\gamma(t_n)$ and $\gamma_i(t_n)$ on the approximate $X_i(t_n)$ through $g_{t_n}(p)$ and $g_{t_n}(p_j)$ resp, that bound a rectangle $R(t_n)\to R(\infty)$.   Since $p_j$ is $2$-sided  the vertical sides would $\eta(t_n)$ interact, where $\eta(t_n)\to 1$.  This would  violate Proposition~\ref{prop:rectangles}.   
   We  finish the proof of (5) in the non-cylinderical case by noting 
that we   can cover $X_i(\infty)$ with open discs. 
  
We see similarly that there cannot be  images of  generic points $q$ and $p_j$ in the same cylinder for a sequence $t_n\to\infty$.  Otherwise  
there would be  two vertical segments of the same  length through these image points bounding a rectangle.   We can take these vertical segments to be longer than the circumference of the cylinder, so that when pulled back to $(X,\omega)$ under $g_{t_n}^{-1}$, their lengths go to infinity.  
This is again a  contradiction. Thus (5)  holds for cylinders.

  Since the set of generic points for ergodic measures has full Lebesgue measure we see that (6) holds. 
\end{proof}
The rest of this section deals with generic points  of measures that are not ergodic.  We wish to associate a subset  for each as we did for ergodic measures. 
This is accomplished in Proposition~\ref{prop:separating}.  The main difficulty is that the set of generic points for generic but not ergodic measures has Lebesgue measure $0$, so we cannot use  the method  in Proposition~\ref{prop:ergodic} to deal with them. 
We continue to use  the notation $X_i(t_n)$ for the components of the $(\epsilon_0,\epsilon_1)$ thick-thin decomposition of $X(t_n)$  associated to ergodic measures. We will denote by $Y(t_n)$ a component of the complement of $\cup_i X_i(t_n)$.

%\begin{lem}
%\label{lem:disc}
%If a component $Y(t_n)$ is a disc, then it does not contain the image $g_{t_n}(q)$ of a generic point $q$ of a nonergodic measure.

%\end{lem}
%\begin{proof}
%The vertical line through $g_{t_n}(q)$ of fixed bounded length would enter a domain $X_j(t_n)$ associated to an ergodic measure and would $\eta$ interact  for some $\eta>0$, with the image of a generic point of an ergodic measure, contradicting (4) of Proposition~\ref{prop:ergodic}.

%\end{proof}

We next divide each  $Y(t_n)$ into $(\epsilon_0,\epsilon_1)$ thick-thin components $Y_j(t_n)$.
As in  Remark~\ref{rem:gauntlimit},
we pass to the limit of $\omega_j(\infty)$ on  $Y_j(\infty)$ of the normalized $\omega_j(t_n)/\lambda(Y_j(t_n))$.  If the sequence is plump, then $Y_j(\infty)$ does not have boundary and has finite area. 

\begin{defin} We say a subset $W(t_n)$ of a component $Y_j(t_n)$ is
{\em associated} to a  generic point $q$ for a ($2$-sided) generic non-ergodic  measure $\nu$   and vertical   segment $\beta(t_n)$ through $g_{t_n}(q)$,     if
$\beta(t_n)\subset W(t_n)$, 
$$e^{\frac{t_n}{2}}|\beta(t_n)|\to\infty\ \text{and}\  |\beta(t_n)|\emul \lambda(Y_j(t_n))$$
and such that exactly one of the following holds. 
\begin{itemize} 

\item $W(t_n)$ is a cylinder whose limit on $Y_j(\infty)$ is a cylinder of $\omega_j(\infty)$.
\item  $W(t_n)$  is a minimal domain of $\omega_j(\infty)$ 
\item   $W(t_n)$ is a union of  horizontal trapezoids  $\text{Trap}_k(t_n)$ each of which has two of its sides on two long boundary components and  
$\beta(t_n)$   crosses the trapezoids and   joins nonzero points on a pair of   short boundary loops of $Y_j(t_n)$.
% and  an  intersecting  horizontal arc  $\sigma(t_n)$ in a trapezoid which  joins
%long boundary loops of $Y_j(t_n)$.  
  %neither of which  is on the boundary of  any  $X_j(t_n)$  associated to an ergodic measure.
\item   $W(t_n)$ is a union of  horizontal trapezoids  $\text{Trap}_k(t_n)$ each of which has two of its sides on two long boundary components of $Y_j(t_n)$ and such that $\cup_k \text {Trap}_k(t_n)$ 
contains  a closed geodesic $\gamma(t_n)$ which is isotopic to the composite of $\beta(t_n)$  and a horizontal arc in one of the trapezoids. The arcs of the trapezoids are not isotopic to a boundary component.  
 
 \end{itemize}
We will refer to these $W(t_n)$ given  above  as type (I), (II), (III) or (IV).
Types (III) and (IV) only can occur in a gaunt sequence. 
\end{defin}

 \begin{prop}
\label{prop:separating}
 
 Suppose there are $m$ generic but non-ergodic invariant measures $\nu_i$ with generic points $q_i$ contained in $Y_j(t_n)$.
For large $t_n$ there are $m$ disjoint subsets    $W_1(t_n),\ldots, W_m(t_n)$ where $W_i(t_n)$ is associated to $q_i$ and vertical $\beta_i(t_n)$ and  for each $i$ there is $j$ such that $W_i(t_n)\subset Y_j(t_n)$.  
 
\end{prop}

%If the seq
\begin{proof}

%opposite direction vertical  segments   $\beta_i(t_n)$ through $g_{t_n}(q_i)$ of length $\delta_0\lambda(Y_j(t_n))$
%s contained in $Y_j(t_n)$
%

First suppose  $Y_j(t_n)$  with corresponding $\omega_j(t_n)$ is a plump sequence.    By  Proposition~\ref{prop:limit} its limit $\omega_j(\infty)$ does not have poles at the punctures. 
Suppose there are $m'$ generic points $q_i$ of distinct non-ergodic generic measures such that $g_{t_n}(q_i)\in Y_j(t_n)$. In this case we will prove that there are $m'$ disjoint $W_i(t_n)$ of types (I) and  (II)  that are subsets of  $Y_j(t_n)$ and are associated to these generic points.

We claim there exists $\delta_0>0$ such that for any generic point $q$,  and large $t_n$,  a vertical segment shorter than $\delta_0 \lambda(Y_j(t_n))$ through $g_{t_n}(q)\in Y_j(t_n)$ does not leave $Y_j(t_n)$ in both directions.

Arguing by contradiction,  suppose there is a sequence $\delta_n\to 0$ such that the vertical segment $\beta(t_n)$ in both directions through $g_{t_n}(q)$ of length at most $\delta_n\lambda(Y_j(t_n))$ hits  boundary  components $\alpha_1(t_n)$ and $\alpha_2(t_n)$ of $Y_j(t_n)$.    By definition of  plumpness  there is $\epsilon_n\to 0$ such that  $|\alpha_i(t_n)|\leq  \epsilon_n\lambda(Y_j(t_n))$. If $\alpha_1(t_n)\neq \alpha_2(t_n)$ then the concatenation 
 $\beta(t_n)* \alpha_1(t_n)*\beta^{-1}(t_n)*\alpha_2(t_n)$ of arcs produces a closed essential curve, with length at most  $(2\delta_n+2\epsilon_n)\lambda(Y_j(t_n))$.  Since $2\delta_n+2\epsilon_n\to 0$, we have a contradiction to the definition of $\lambda(Y_j(t_n))$.  
 If $\alpha_1(t_n)=\alpha_2(t_n)$  there is a similar surgery using a segment of $\alpha_1(t_n)$ joining the endpoints of $\beta(t_n)$.  This  proves the claim.  

%First we include the v
%rtical $\beta_i(t_n)$ joining short boundary components  that are not boundary components of $X_k(t_n)$.    Suppose there are $m''\leq m'$ of them. 
 Since for all boundary components $\alpha(t_n)$, $\frac{\lambda(Y_j(t_n))}{|\alpha(t_n)|}$ is bounded away from $0$,  
the claim, and (\ref{eq:lengths}) imply  that $$e^{t_n/2}|\beta(t_n)|\to\infty.$$

We now take the horizontal foliation of $Y_j(\infty)$ defined by $\omega_j(\infty)$ and the limits $\beta_i(\infty)$  of the  $\beta_i(t_n)$.    Each $\beta_i(\infty)$ contains a subsegment $\beta_i'(\infty)$ of length proportional to $\beta_i(\infty)$ such that $\beta_i'(\infty)$ is contained in a horizontal domain $W_i(\infty)$; either a cylinder or a minimal domain.  
  
  Now we claim that    
    $\beta_i'(\infty)$ and $\beta_j'(\infty)$ cannot be contained in the same domain for $i\neq j$. For otherwise  in the minimal case choosing  them as cross sections for the horizontal line flow, we would find  segments of each of length comparable to $\beta_i(\infty)$ that bound a rectangle, which is impossible by
     Proposition~\ref{prop:rectangles}.  The same argument holds for cylinders.  Thus the approximate 
$Y_j(t_n)$  then  contains $m'$  disjoint subsurfaces $W_i(t_n)$, which finishes the proof in   
the plump case. 

Next suppose  the sequence $Y_j(t_n)$ is gaunt.   
Again take its limit $Y_j(\infty)$ with $1$-form $\omega_j(\infty)$.   Now there are long boundary components so the  limiting $Y_j(\infty)$  has boundary.   Let $B_0$ the sum of the lengths of the boundary components of $Y_j(\infty)$. 
 Unlike the plump case, for any fixed $\delta_0$ it is possible as  $t_n\to\infty$ there might be a vertical segment that  might cross boundary  components in both  directions while being shorter than $\delta_0\lambda(Y_j(t_n))$. 
    
   Now consider a vertical segment $\beta_i(t_n)$ through a limit point of  $g_{t_n}(q_i)\in Y_j(t_n)$ of the generic point  $q_i$ which after renormalization has  length $2B_0$.   Suppose it remains in $Y_j(\infty)$.   The first possibility is it enters a cylinder 
or minimal domain that comes from the horizontal line flow. If that happens we associate the approximate  domain to it just as in the plump case.  
%This allow for the possibility that it joins two short boundary segments since its length is bounded below by $\delta_0$. 

%Suppose it crosses two short within $2B_0\lambda_j(t_n)$. Then use it. It has length at least $\delta_0\lambda_j(t_n)$.  
% Take the decomposition of $Y_j(\infty)$ into trapezoids, cylinders, and minimal domains. 
 
 The next possibility is a subsegment of $\beta_i(t_n)$ joins two short boundary segments and does not enter  a cylinder or minimal domain. Its length is at least comparable to $\lambda_j(t_n)$.   It is  entirely contained  in a union of  trapezoids $W_i(t_n)$.

This gives the type (III) possibility.  No other $\beta_k(t_n)$ may have a subsegment contained in any of the same trapezoids intersected  by $\beta_i(t_n)$. Otherwise $\beta_i(t_n)$ and $\beta_k(t_n)$ would $\eta$ interact for some $\eta>0$, which again is impossible by    Proposition~\ref{prop:rectangles}.

The next possibility is that  the limiting  $\beta_i(\infty)$ is  still contained in a union of trapezoids, but does not join two short boundary loops.   
%that in at least one direction, $\beta_i(t_n)$ remains in $Y_j(t_n)$ and does not join short boundary components. Taking a limit as $t_n\to\infty$ and renormalizing by dividing lengths by $\lambda(Y_j(t_n))$ we find a vertical segment of length $2B_0$ that lies in $Y_j(\infty)$.%For if it did then $\beta_i(t_n)$ and $\beta_k(t_n)$ would $\eta$-interact for some $\eta>0$, again contrary to Proposition~\ref{prop:rectangles}.  
%We associate to the union of trapezoids crossed by $\beta_i(t_n)$ this generic point $q_i$.
Since   the sum of distances across the union of trapezoids is at most $B_0$ and  $\beta_i(\infty)$ has length $2B_0$, it  must return to some trapezoid. Closing the path up in that trapezoid using a horizontal segment  we produce   a closed loop,  denoted  $\gamma_i(\infty)$.   %
%The closed loop cannot be isotopic to a boundary component; for otherwise there would be a cylinder in the homotopy class of the boundary component, contrary to assumption.  
The loop $\gamma_i(\infty)$ and the trapezoids are approximated by loops $\gamma_i(t_n)$  and approximating  trapezoids $\text{Trap}(t_n)$.  We let $W_i(t_n)$ be again the union of trapezoids. This gives a type (IV)  $W_i(t_n)$.  As is the case of type (III) sets $W$, no other $\beta_k(t_n)$ may have a subsegment contained in any of the same trapezoids intersected  by $\beta_i(t_n)$.
%The loop  has the property that it cuts every trapezoid it crosses into a pair of trapezoids.

Now assume none of the previous possibilities hold so the  
  the vertical segment   $\beta_i(t_n)$ through $g_{t_n}(q)$  of lengths $2B_0$  hits a boundary  loop of $Y_j(t_n)$  in both directions, at least one of which, denoted $\alpha(t_n)$  is long.   Since the length $|\alpha(t_n)|$ is 
comparable to $\lambda(Y_j(t_n))$,  the estimate (\ref{eq:modulus}) shows the modulus of the corresponding expanding annulus of $\alpha(t_n)$ inside $Y_j(t_n)$ that $\beta_i(t_n)$ crosses is uniformly bounded. Then there must either be a cylinder $Y_\ell(t_n)$ of  modulus going to infinity as $t_n\to\infty$ isotopic to $\alpha(t_n)$, or an expanding annulus of  modulus going to infinity in a neighboring $Y_\ell(t_n)$ of the thick-thin decomposition.   If there is a cylinder $Y_\ell(t_n)$ with modulus  going to infinity we can associate to $\beta_i(t_n)$ that  cylinder and denote it by $W_i(t_n)$. If there is not such a cylinder, but rather an expanding annulus in a neighboring $Y_\ell(t_n)$ then
  (\ref{eq:modulus}) 
 %  (\ref{eq:boundedbelow}) 
   shows that as $t_n\to\infty$,   $$\frac{\lambda(Y_\ell(t_n))}{\lambda(Y_j(t_n))}\to\infty.$$    
   This implies that $\alpha(t_n)$ is a short loop of $Y_\ell(t_n)$.  
   
   We  extend $\beta_i(t_n)$ to have  length $2B_0\lambda(Y_\ell(t_n))$. If it remains in  $Y_\ell(t_n)$, then as in the previous discussions,  we associate the corresponding subsurface  $W_i(t_n)\subset Y_\ell(t_n)$ to $\beta_i(t_n)$ and then the vertical segment through any other generic point cannot intersect $W_i(t_n)$.    If $\beta_i(t_n)$  intersects a short boundary loop $\alpha(t_n)$ of $Y_\ell(t_n)$,  then exactly as before, a subsegment of $\beta_i(t_n)$ of length at least $\delta_0\lambda(Y_\ell(t_n)$ joins the short boundary   components of $Y_\ell(t_n)$, and is associated to a $W_i(t_n)$.   Suppose finally that $\beta_i(t_n)$  hits a long boundary loop of $Y_\ell(t_n)$.  Then exactly as in the previous paragraph it must then enter a cylinder of big modulus or enter a  neighboring $Y_k(t_n)$ of the thick-thin decomposition with $\frac{\lambda(Y_k(t_n))}{\lambda(Y_\ell(t_n))}\to\infty$.  If it is a cylinder with modulus going to infinity, that is the associated subset. 

If not, 
and rather it enters a neighboring $Y_k(t_n)$, then we repeat the process by extending the vertical segment in $Y_k(t_n)$.  The condition of entering a neighboring component of the thick-thin decomposition with increasing size $\lambda$ can only occur a bounded number of times, in terms of the genus.    
Therefore after a bounded number of steps this process must terminate with some $W_i(t_n)$ associated to $\beta_i(t_n)$.

 \end{proof}
%and the vertical  line through $g_{t_n}(\tilde q_i)$ contained in  $X_i(t_n)$.  

The following Lemma  gives additional restrictions on sets $W$ for different measures.
Recall  $\emul$ means up to uniform multiplicative constants independent of $t_n$. 
They all will follow from   Proposition~\ref{prop:rectangles}.
First we need a notion of closeness of interior zeroes  to the boundary.

\begin{defin}
An interior zero  $z(t_n)$ of  $Y_j(t_n)$ is close to the short loop $\alpha(t_n)$  on the boundary of $Y_j(t_n)$ if a  vertical separatrice $\beta(t_n)$ leaving $z(t_n)$ hits $\alpha(t_n)$ and $$\lim_{t_n\to\infty} \frac{|\beta(t_n)|}{\lambda(Y_j(t_n))}\to 0.$$
\end{defin}

\begin{lem}
\label{lem:interact}

For large $t_n$ 
\begin{enumerate}
\item a pair of gaunt sequences   $Y_k(t_n),Y_\ell(t_n)$ cannot have vertical segments $\beta_k(t_n),\beta_\ell(t_n)$ associated to type (III) sets $W_k(t_n)$ and $W_\ell(t_n)$   that have endpoints in $J_k(t_n)\cap J_\ell(t_n)$ where these are  short boundary segments on a common boundary loop. 
\item 
If  $Y(t_n)$ is a component of  the complement of the union of domains $X_j(t_n)$ associated to ergodic measures  there is no type (III) $W(t_n)$ with associated vertical segment $\beta(t_n)$ with an endpoint on $\partial Y(t_n)$. 
\item 
If  $W_1(t_n),W_2(t_n)\subset Y_j(t_n)$ are type (III) sets with associated vertical segments $\beta_1(t_n),\beta_2(t_n)$ with endpoints on the same short boundary loop $\gamma(t_n)$  then either they are separated by a zero on $\gamma(t_n)$ or there is a zero close to the segment bounded by the endpoints. 
%  \item  If $W_1$ is type (IV) consisting of trapezoids $\text{Trap}(t_n)$ and $\beta(t_n)$ then no other $\beta'(t_n)$ may enter these $\text{Trap}(t_n)$.

\end{enumerate}

  \end{lem}

\begin{proof}

We argue by contradiction. For (1)  without loss of generality assume  
$\lambda(Y_k(t_n))\geq \lambda(Y_\ell(t_n))$.  Extend $\beta_\ell(t_n)$  into $Y_k(t_n)$. For a fixed $\eta>0$ it will $\eta$ interact with $\beta_k(t_n)$ for large $t_n$,  which is a contradiction.    
For (2),  if there were such a vertical segment,  extending the vertical segment into $\cup X_j(t_n)$ it would interact with an ergodic vertical segment also contradicting  (4) of Proposition~\ref{prop:ergodic}. 
For the proof of  (3), if there is no such zero  then the vertical  segments would $\eta$ interact for some fixed $\eta>0$ and large $t_n$. \end{proof}
%\end{proof} 

\section{Proof  of Theorems}

 Now for the rest of the paper, by passing to subsequence of $t_n$ we can assume that the topology and combinatorics  of the ergodic components $X_j(t_n)$, complementary components  $Y(t_n)$,  thick thin components  $Y_j(t_n)$ of $Y(t_n)$ and $W_i(t_n)$  are all independent of $t_n$ in the sense that there are homeomorphisms of the underlying surface that preserve these subsets.   
  We will  write $Y_j(\infty)$ for limit of $Y_j(t_n)$.  
 
 \begin{defin}
For any  $Z=Z(t_n)$, each component of which is a union  of $Y_j(t_n)\subset Y(t_n)$ in the thick-thin decomposition glued together along  boundary loops, let $m(Z)$ be the maximum number of $W$ of type (I)-(IV) that can be  contained in $Z$.
\end{defin}

We define a quantity $\rho(Z)$ for any such $Z$. 

\begin{defin}
\label{def:rho}
For any  $Z$, %component $Y_j=Y_j(t_n)$ of the thick thin decomposition of $S\setminus X(t_n)$,  suppose
 let   $g(Z)$ be the genus, $s(Z)$ the number of interior zeroes, and $n(Z)$ the number of boundary components. 
 Define  $$\rho(Z)=g(Z) +s(Z)+ n(Z)-1.$$ 
\end{defin}

\begin{rem}
We can make the same definition for any $\Omega\subset Z$ a domain with boundary which is a union of saddle connections.
\end{rem}
 
 The next two propositions together will immediately give the main theorem.
 
 \begin{prop}
\label{prop:sum} Let $Y(t_n)$ be a component of the complement of the union of the ergodic components $X_j(t_n)$. Then $k+\sum_{Y(t_n)} \rho(Y(t_n))\leq \rho(S)=g+s-1$, where  $k$ as before is the number of ergodic probability measures and $g$ is the genus of the entire surface $S$ and $s$ is the number of zeroes of $\omega$. 

\end{prop}

\begin{prop}
\label{prop:basicY}
 For $Y(t_n)$  a component of the complement of the ergodic components, 
 $m(Y(t_n))\leq \rho(Y(t_n))$.
 
 \end{prop}

\begin{proof} [Proof of Theorem~\ref{thm:main}]

Combine Proposition~\ref{prop:sum} with Proposition~\ref{prop:basicY}.

\end{proof}

\begin{proof}[Proof of Proposition~\ref{prop:sum}]
We successively remove the domains $X_j(t_n)$. At each stage we maximize the value $\rho$ of the complement if the domain we remove $X_j(t_n)$ is a cylinder.  Inductively,  given a collection of cylinders $X_j(t_n)$ with complement  $Y'$ we remove a cylinder  giving  a new complement $Y''$.  We continue $k$ times until we are left with $Y$. 
It is enough to show 
 $\rho(Y'')<\rho(Y')$ each time we remove a cylinder. 

We have $n(Y'')\geq n(Y')$ but $n(Y'')-n(Y')\leq s(Y')-s(Y'')$ since each new boundary component must have at least one zero on it. If we also
decrease the genus, then we see that  $\rho(Y'')-\rho(Y')\leq -1$.

Suppose then  the genus of $Y''$ equals the genus of $Y'$.    Thus $Y''$ has  one more  connected component than $Y'$.   If the complementary connected components are not discs, then  the extra $-1$ in the definition of $\rho$ again gives $\rho(Y'')\leq \rho(Y')-1$.  If a complementary  component $\hat Y$ is a disc then $\rho(\hat Y)=0$  Then ignoring the discs $n(Y'')=n(Y')$, but again since  $Y''$ has at least one fewer interior zero so again $\rho(Y'')\leq \rho(Y')-1$.  Iterating this reduction $k$ times until all cylinders  $X_j(t_n)$ have been removed we have proved the desired formula $k+\sum_Y \rho(Y)\leq \rho(S)=g+s-1$. 
\end{proof}

 We will prove 
 Proposition~\ref{prop:basicY}
in two steps. In the first step, Proposition~\ref{prop:basic},
we bound $m(Y_j(t_n))$.
%
 %for $Z(t_n)$ where $Z(t_n)$ consists of subsets of unions of $Y_j(t_n)$ glued along bound
%We now make definition of quantities that will go into the count $m(Z(t_n))$.
The main complication are the type (III) sets $W$ and accounting for zeroes on $\partial Z(t_n)$.
In the second step  we  will glue the   $Y_j(t_n)$ together  to form $Y$. 

%% \end{defin}
We now fix a domain $Y_j(t_n)$. 
Suppose we have a collection $\mathcal{T}$ of type (III) sets $W\subset Y_j(t_n)$. Each determines a pair of segments along loops of $\partial Y_j(t_n)$ and a vertical line $\beta$ with endpoints on those segments. 

\begin{defin}
Let $\mathcal{G}$ be a graph with vertices that are the boundary components $C(t_n)$  of $Y_j(t_n)$ that contain a vertical segment $\beta$ associated to a type (III)  
$W$. The  edges $\mathcal{E}$ are the  $\beta$.  
\end{defin}

\begin{defin}
Let  $\mathcal{G}_1$ be a possibly disconnected  subgraph of $\mathcal{G}$ with  edges $\mathcal{E}_1$ such that each connected component of  $\mathcal{G}_1$   does not contain any  loops and  $\mathcal{G}_1$  is maximal in that $\text{card}(\mathcal{E}_1)$ is as large as possible. 
\end{defin}

 \begin{defin}
%For each short boundary loop $\gamma$ let $\tau(\gamma(t_n))$ be the number of edges  $\beta\in \mathcal{G}\setminus \mathcal{G}_1$ that have an endpoint on $\gamma(t_n)$. 
For each  component $B(t_n)$ of $\partial Y_j(t_n)$, let  $\tau(B(t_n))$ be the number of edges $\beta\in\mathcal{G}\setminus \mathcal{G}_1$ that   have an endpoint on $B(t_n)$ and  let 
 $$\tau(Y_j(t_n))=\sum \tau(B(t_n)),$$ the sum over the boundary components $B(t_n)$ of $Y_j(t_n)$.

\end{defin}

\begin{defin}
Let $s''(Y_j(t_n))$ be the number of close zeroes to short loops on $\partial Y_j(t_n)$ that contain endpoints of tye (III) $\beta\subset W$.
% and finally let $\tau(\gamma)=n(\gamma)+s''(\gamma)-1$. $\tau(Z(t_n))=\sum_{\gamma} \tau(\gamma)$

\end{defin}

 By definition,  $$\tau(Y_j(t_n))=\text{card}(\mathcal{T})-\text{card}(\mathcal{E}_1).$$ 

\begin{rem}

\label{rem:only}
Since $\mathcal{G}_1$ has maximal cardinality, then for any  vertex $v=B(t_n)$ either $v$ is a vertex of some edge in $\mathcal{E}_1$   or the only edges with $v$ as one endpoint join $v$ to itself.

\end{rem}

\begin{prop}
\label{prop:basic}  For large $t_n$
%any $Z(t_n)\subseteq Y(t_n)$ that is either a single $Y_j(t_n)$ or formed
%by gluing a collection of $Y_j(t_n)$ together satisfies 
 $$m(Y_j(t_n))\leq \rho(Y_j(t_n))+\tau(Y_j(t_n))-s''(Y_j(t_n))$$ 
\color{black}
\end{prop}

  \begin{proof}
  We consider first  the case that $Y_j(t_n)$ is a plump sequence.  In this case $s''(Y(t_n))=\tau(Y_j(t_n))=0$
 since there are no type (III) sets $W$. Inductively   starting with $Y_j(t_n)$ we remove cylinders and minimal domains, denoted by $U$,  leaving a surface $\Omega$ with corresponding $\rho(\Omega)=g(\Omega)+s(\Omega)+n(\Omega)-1$. The proposition follows from the claim that when we remove another $U$ that $$\rho(\Omega\setminus U)\leq\rho(\Omega)-1.$$ We prove the claim.  If $U$ is a cylinder then the number of additional boundary components of $\Omega\setminus U$ is offset by the fewer number of interior zeroes. If the genus  of $\Omega\setminus U$ is smaller than the genus  of $\Omega$ we have the claim. If the genus is the same, then the cylinder is separating and $\Omega\setminus U$ has one more connected component than $\Omega$ and so there is an additional $-1$ in the definition of $\rho$ and again we have the claim. 
 The situation of adding a minimal domain is similar.

Next consider the case of a gaunt sequence.  
 Consider the  vertical segments $\beta(t_n)$ of all type (III) sets $W$. They limit on vertical  saddle connections $\beta(\infty)$ joining nodes in $Y_j(\infty)$.  We denote the complement of these vertical saddle  connections in $Y_j(\infty)$ by $Y_j^F(\infty)$. We may define $\rho(Y_j(\infty))$ in the same way as before, where in the definition, the nodes or punctures are counted as interior zeroes in $s(Y_j(\infty))$.

%Let $m'(Y_j(t_n))$ the number of type (III) $W$.
 \begin{lem}

 \label{lem:first}
For large $t_n$, 
$\text{card} (\mathcal{T})\leq 
 \rho(Y_j(t_n))-\rho(Y_j^F(\infty))+\tau(Y_j(t_n))-s''(Y_j(t_n))$.

 \end{lem}
\color{black}

\begin{proof}
%The equality on the left follows  since  the difference $m(Y_j(t_n))-m(Y_j^F(\infty))$ is exactly the set of measures that give rise to type (III) sets $W$.
%We write the right side as $$\rho(Y_j(t_n))+e(Y)-\rho(Y_j^m(\infty))=\rho(Y_j(t_n))+e(Y) -\rho(Y_j(\infty))+\rho(Y_j(\infty)-\rho(Y_j^m(\infty)).$$
%We prove (\ref{eq:first}). 
We start with $Y_j(t_n)$ and first remove the collection of edges $\beta(t_n)\in \mathcal{E}_1$.  Then  remove the rest of $\beta(t_n)\in\mathcal{E}$.   At each stage we have a set $\Omega(t_n)$ whose boundary consists of segments of $\partial Y_j(t_n)$ and  vertical $\beta(t_n)$ of $W$ that have been removed. We then have a new $W$ and vertical $\beta'(t_n)$ and set 
  $\Omega'(t_n)=\Omega(t_n)\setminus \beta'(t_n)$. For edges $\beta(t_n)\in \mathcal{E}_1$, at each stage  $\beta(t_n)$ will join distinct boundary components which says that $\Omega'(t_n)$ has fewer boundary components than $\Omega(t_n)$ which implies    \begin{equation}
  \label{eq:strict}
  \rho(\Omega'(t_n))<\rho(\Omega(t_n)).
  \end{equation} For  $\beta(t_n)\notin \mathcal{E}_1$ we may increase the number of boundary components, but if so,   either we decrease the genus or increase the number of connected components and therefore the number of $-1$ in the definition of $\rho(\Omega')$ compared to $\rho(\Omega)$.  In either case we still have $$\rho(\Omega'(t_n))\leq\rho(\Omega(t_n)),$$ but not necessarily strict inequality. 
  
 Removing all $\beta(t_n)$  we end with some final $\Omega_j^F(t_n)$.  As remarked earlier, the $\beta(t_n)\subset Y_j(t_n)$ limit onto $\beta(\infty)\subset Y_j(\infty)$ joining the nodes.   The segments of $\partial Y_j(t_n)$ on the boundary of $\Omega_j(t_n)$  as well as the   close zeroes  limit to the nodes  We conclude that  
$$\rho(Y_j^F(\infty))=\rho(\Omega_j^F(t_n))-s''(t_n).$$   Now 
$\text{card}(\mathcal{T})=\text{card}(\mathcal{E}_1)+\tau(Y_j(t_n))$.
By (\ref{eq:strict}) $\text{card}(\mathcal{E}_1)\leq \rho(Y_j(t_n))-\rho(\Omega_j^F(t_n))$.
 The lemma follows.

\end{proof}

 \begin{lem}
\label{lem:second}
$m(Y_j^F(\infty))\leq  \rho(Y_j^F(\infty))$.
\end{lem} 
\begin{proof}

The proof begins the same way as the proof in the plump case. 
We remove all $W$ that are cylinders and minimal domains  decreasing $\rho$ by at least that number of domains until we are left with a possibly disconnected  $\Omega(t_n)$ and need to consider type (IV) sets $W\subset \Omega(t_n)$.  Each $W$ is a union of  trapezoids $\text{Trap}(t_n)$ and contains a closed loop $\gamma(t_n)$.  
The first possibility is that cutting along $\gamma(t_n)$ reduces the genus of the surface.  In addition it reduces the number of interior  zeroes.  One adds a pair of boundary components. At this stage the value of $\rho$ has not been reduced. However we also remove the arcs in the trapezoid crossing $\gamma(t_n)$. Since the arcs join different boundary components, cutting along the arc reduces their number and hence reduces $\rho$. 

The other possibility that is cutting along $\gamma(t_n)$ does not reduce the genus but divides the surface into two components. There must be a zero on $\gamma(t_n)$ and a pair of $-1$ in the definition of $\rho$ as compared to a single $-1$ to begin. There are $2$ additional boundary components. Thus the corresponding $\rho$ are the same. But now an arc of the trapezoid joins  boundary components of the surface to the two new two components found by  cutting along $\gamma(t_n)$.
Cutting along this arc reduces the total number of boundary components, reducing the value of $\rho$.  Finally the same analysis holds if $\gamma(t_n)$ is isotopic to a boundary component since the trapezoid arc cannot be isotopic to the boundary.  
 We conclude that each time we remove a set $W$ the value of $\rho$ strictly decreases which is the desired result. 

\end{proof}
We finish the proof of Proposition~\ref{prop:basic}.
We first note that by construction $$m(Y_j(t_n))=\text{card}(\mathcal{T})+ m(Y_j^F(\infty)).$$
Combining this with the inequalities from Lemma~\ref{lem:first} and Lemma~\ref{lem:second}  gives 
\begin{equation}
\label{eq:total}
m(Y_j(t_n))\leq \rho(Y_j(t_n))+\tau(Y_j(t_n))-s''(t_n).
\end{equation}
%On the other hand close zeroes limit to nodes so
%\begin{equation}
%\label{eq:decrease}\rho(Y_j(t_n))-\rho(Y_j(\infty))\geq s''(Y_j(t_n)).
%\end{equation}

\end{proof}
 
\begin{proof}[Proof of Proposition~\ref{prop:basicY}]

We do not need to record dependence on time $t_n$ in this proof.
Consider boundary loops of  any  $Y_j$ which is either plump or  gaunt but the loops  are not subsets of boundary components which are vertices  of edges  in $\mathcal{G}\setminus \mathcal{G}_1$.  We start by considering the case of gluing $Y_j$  to itself along two such loops to form $Y_j'$.  We have $$s(Y_j')\geq s(Y_j)+1, g(Y_j')=g(Y_j)+1, n(Y_j')=n(Y_j)-2.$$   This gives  $\rho(Y_j')\geq \rho(Y_j)$
so   $m(Y_j')=m(Y_j)$ and $\rho(Y_j)\leq \rho(Y_j')$. If  there were terms $\tau(Y_j)$ or $s''(Y_j)$ in the bound for $m(Y_j)$, given in Proposition~\ref{prop:basic} they are not effected by the gluing  and so appear in the expression for $m(Y_j')$, and we conclude using  Proposition~\ref{prop:basic}  that \begin{equation}
\label{eq:gluing}m(Y_j')\leq \rho(Y_j')+\tau(Y_j)+s''(Y_j).
\end{equation}

 Next we consider that $Y_j$  is glued to itself,  but a cylinder is inserted. We have again $g(Y_j')=g(Y_j)+1$ but the possibility that a cylinder is associated to a generic measure gives  $m(Y_j')=m(Y_j)+1$. If the boundary circles of the cylinder are not joined at a point, then $s(Y_j')\geq s(Y_j)+2$ and $n(Y_j')=n(Y_j)-2$.   Putting these facts together we  have $\rho(Y_j)+1\leq \rho(Y_j')$. 
 If the two circles  are joined at a point, then $s(Y_j')=s(Y_j)+1$ and $n(Y_j')=n(Y_j)-1$, so we again have
$\rho(Y_j)+1\leq \rho(Y_j')$. We conclude again using  Proposition~\ref{prop:basic} that
(\ref{eq:gluing}) holds.

 Next  consider distinct surfaces $Y_j$ and $Y_k$ glued together to form a new $Z$ so $m(Z)=m(Y_j)+m(Y_k)$.  Again assume that the assumptions described in the first paragraph hold. Now  $\rho(Y_j)+\rho(Y_k)$ has a $-2$ in its definition, while $\rho(Z)$ contains a single $-1$ term.  We also have $$g(Z)=g(Y_k)+g(Y_k), n(Z)=n(Y_j)+n(Y_k)-2, s(Z)\geq s(Y_k)+s(Y_k)+1,$$ 
giving $$\rho(Z)\geq \rho(Y_j)+\rho(Y_k).$$   If a cylinder is inserted  between $Y_j$ and $Y_k$ then $m(Z)=m(Y_j)+m(Y_k)+1$, but now $s(Z)\geq s(Y_j)+s(Y_k)+2$.
In either case (\ref{eq:gluing}) holds for $Z$. 
%The rest of the terms are the same so again we have (\ref{eq:combine1}) in this case.
 
 We continue this process of gluing. 
 If there are no  gaunt $Y_j$, or there are gaunt $Y_j$, but $\tau(Y_j)=0$ for all of them, then at the end of the gluing we conclude $m(Y)\leq \rho(Y)$ and we are done.
 Notice if a loop of $Y_j$ is glued to a loop on the boundary of the ergodic domains, it cannot contain  the endpoint of $\beta$ of type (III) by(2) of  Lemma~\ref{lem:interact}.

 Thus assume that $\tau(Y_j)>0$ for some $Y_j$.
We assume we have performed all possible gluings described in the previous paragraphs leaving  surfaces, denoted $Z_k$ composed of a number of $Y_j$. They have   $n_k'$ boundary components which contribute to $\tau(Z_k)$ and $n_k''$  boundary components in common with the boundary of the ergodic components.   Consider then a boundary component $B$ of some $Y_j\subset Z_k$ and such that  $\tau(B)>0$. Suppose    $\gamma_1,\ldots, \gamma_p\subset B$ are the short boundary loops that each contain endpoints of a vertical segment $\beta$ associated to some $W\in \mathcal{T}$.  On each $\gamma_i$ there is a collection of disjoint segments, each containing an endpoint of exactly one of these $\beta$.  The  endpoints of the segments themselves are either zeroes or points on vertical lines that hit a close zero.  By Lemma~\ref{lem:common}  the $\gamma_i$ do not share any zeroes, and as a consequence do not share close boundary zeroes with each other.      Suppose $Y_j$ is glued along  each  $\gamma_i$ to some $Y_\ell\subset  Z_m$ for some $Z_m$.   By Lemma~\ref{lem:interact}  the segments  along $\gamma_i$ that correspond to type (III) $W\subset Y_j$  are disjoint from the possible segments  from the type (III) $W\subset Y_\ell$.  Let  $a_j(\gamma_i)$ be the number of the segments on $\gamma_i$ corresponding to $W\subset Y_j$ and $a_j(B)=\sum_{\gamma_i\subset B}a_j(\gamma_i)$. 
 % \label{eq:notstrictcount}
  %a_i+b_i\geq \text{card}(\gamma_i,Y_j,W)+\text{card}(\gamma_i,Y_i,W).
  %\end{equation}
  %Here $\text{card}(\gamma_i,Y_j,W)$  denotes the number of $W\subset Y_j(t_n)$ of type (III)  with $\beta$ an endpoint on $\gamma_i$,  with a similar definition for $\text{card}(\gamma_i,Y_i)$.  

  Since $\tau(B)>0$, 
  by Remark~\ref{rem:only} one of two possibilities holds.
  The first is that  $B$ is a vertex of $\mathcal{G}_1$, in which case  there is a $W$ of type (III) with vertical $\beta$  having endpoint on  $\gamma_i$ but the edge $\beta\in \mathcal{G}_1$ and so this $\beta$  does not contribute to the set of $W$ that contribute to counting $\tau(Y_j)$.  This then says that $a_j(\gamma_i)$ is larger than the number of $\beta$ that do contribute to $\tau$ with endpoints on $\gamma_i$. 
  % which in turn gives    %a+b>\tau(C,Z)+\sum_i \tau(\gamma_i,Z_i)

 The second possibility is that $B$ is not a vertex of $\mathcal{G}_1$. In this case all vertical   $\beta$ of type (III) with endpoints on $B$ must join $B$ to itself. Thus there are two segments on $B$ associated to a single $\beta$ and again the number of segments is larger than the number of $\beta$ contributing to $\tau(B)$. 
We conclude that
   \begin{equation}
  \label{eq:strictcount}
a_j(B)>\tau(B).
\end{equation}

Summing the above inequality over the boundary components $B$ of $\partial Y_j$ we see 
 $$a(Y_j):=\sum_{B\subset \partial Y_j} a_j(B) \geq \tau(Y_j)+n'(Y_j).$$
  
  After all the gluings to form $Y$   we compute the number of interior zeroes $s(Y)$  of $Y$.
  For each $Y_j$ the number of segments considered above is the same as the number of zeroes on the boundary plus the number of close zeroes.   The former become new interior zeroes of $Y$ while the latter were already interior zeroes of $Y_j$.  We conclude that  $$s(Y)=
  \sum_j s(Y_j)+\sum_j \left(a(Y_j)-s''(Y_j)\right)\geq \sum_j \left(s(Y_j)+\tau(Y_j)+n'(Y_j)-s''(Y_j)\right).$$ 
This and the definition of $\rho$ implies $$\sum_j (\rho(Y_j)+\tau(Y_j)-s''(Y_j))\leq \sum_j (g(Y_j)+s(Y_j)+n'(Y_j)+n''(Y_j)-1+\tau(Y_j)-s''(Y_j))\leq g(Y)+s(Y)+n''(Y)-1=\rho(Y).$$
This together with Proposition~\ref{prop:basic} gives 
$$m(Y)=\sum_j m(Y_j)\leq \sum_j (\rho(Y_j)+\tau(Y_j)-s''(Y_j))\leq \rho(Y)),$$
and we are done. 
  
    \end{proof}

\begin{proof}[Proof of Theorem~\ref{thm:special}]
 
  The collection of $X_j(t_n)$ associated to ergodic measures, together with $Y(t_n)$ consists of  at least $g+1$ surfaces.   We claim that the $X_j(t_n)$ must all be cylinders.    The boundary of $X_j(t_n)$ cannot be a single loop joining a zero to itself since it would be trivial in homology, and that is impossible for a translation surface.  If $X_j(t_n)$  is not a cylinder, the complement $S\setminus X_j(t_n)$ either has  $s'=0$ interior zeroes,  genus $g'\leq g-1$ and $n'=1$ boundary components, or  genus $g'\leq g-2$, no interior zeroes, and $n'=2$ boundary components. In either  of the two cases, we have $$\rho(S\setminus X_j(t_n))\leq g'+n'+s'-1\leq g-1.$$ But  this contradicts that $S\setminus X_j(t_n)$  contains at least $g$  subsurfaces, to account for the $g+1$ generic measures. 
This proves the claim.   The areas of the cylinders $X_j(t_n)$ are bounded below by $A_0>0$.   
 
  Now suppose the theorem does not hold so there is a $1$-sided generic non-ergodic measure. We have associated to the $g$ ergodic measures $g$ disjoint cylinders   $X_j(t_n)$ whose boundaries have lengths going to $0$ as $t_n\to\infty$.   Since twisting in these cylinders are independent in cohomology, the saddle connections crossing them are $g$ independent elements  in the relative homology $H_1(S,\Sigma,\mathbb{Z})$ where $\Sigma$ is the set of either $1$ or $2$ zeroes. This space has dimension either $g$ or $g+1$.  Consequently the complement of the $g$ cylinders,  $Y(t_n)=g_{t_n}(X,\omega)\setminus \cup_{j=1}^gX_j(t_n)$   is either a disc in the case of a single zero or an annulus  in the case of $2$ zeroes.    It cannot be a disc, so we have a contradiction in the case of a single zero.  
% and be associated to a generic measure for a bounded length vertical line through the ($1$-sided) generic point of the non-ergodic measure would enter some $X_j(t_n)$ and $\eta$-interact with an ergodic measure for some fixed $\eta$, which is impossible.  Thus $Y(t_n)$ is an annulus.
 
 Thus assume there are a pair of zeroes. 
  For large $t_n$ the  image of the generic point $q$ for a  non-ergodic measure $\nu$  must lie in the annulus $Y(t_n)$.  
We then have $g+1$ cylinders with circumferences going to $0$ along a sequence $t_n\to\infty$.  We will arrive at a contradiction, For fixed cylinder $X_j(t_n)$ with boundary $\alpha_j(t_n)$ and fixed time $t_n$, we have $\lim_{t\to\infty}|g_{t}(\alpha_j(t_n))|_{\omega(t)}\to\infty$.   That is, for any fixed $\alpha_j(t_n)$, its length goes to infinity as $t\to\infty$.  Here we are measuring lengths with respect to  $\omega(t_n)$. Consequently, for each fixed $t_n$ there is  some $j$ and smallest $s_n> t_n$ such that $|g_{s_n}(\alpha_j(t_n))|_{\omega(s_n)}=\sqrt A_0$.
Since the area of $X_j(t_n)$ is at least $A_0$, and the circumference is $\sqrt A_0$, the distance across  $X_j(t_n)$   is therefore at least $\sqrt A_0$.  Thus all short circumference cylinders must be contained in the complement of $ X_j(t_n)$ at time $s_n$.  The complement of $X_j(t_n)$ still must contain $g+1$ cylinders to account for the cylinders associated to the ergodic measures and the non-ergodic generic measures.This is impossible since  the corresponding quantity $g'+n'+s'-1$ is at most $g$. 

\end{proof}

 \end{document}